\documentclass[runningheads]{amsart}

\usepackage{alltt,amssymb,latexsym,amsmath,subfigure}%
\usepackage{amsfonts}
\usepackage{graphicx}

\usepackage[pdftitle={Geometric structure-preserving optimal control of the rigid body},pdfauthor={},
pdfsubject={Geometric structure-preserving optimal control of the rigid body},pdfkeywords={},
bookmarks=false,bookmarksopen=false,pdfstartview={FitH},colorlinks,
    citecolor={blue},%
    filecolor={blue},%
    linkcolor={blue},%
    urlcolor={blue}]{hyperref}

\pdfpagewidth=\paperwidth

\pdfpageheight=\paperheight

    \newtheorem{thm}{Theorem}[section]

    \newtheorem{cor}{Corollary}[section]

    \newtheorem{prb}{Problem}[section]
    \newtheorem{rem}{Remark}[section]

\newcommand{\Jcal}{\mathcal{J}}
\newcommand{\Rcal}{\mathcal{R}}

\newcommand{\etab}{\boldsymbol{\eta}}
\newcommand{\varsigmab}{\boldsymbol{\varsigma}}
\newcommand{\phib}{\boldsymbol{\varphi}}
\newcommand{\xib}{\boldsymbol{\xi}}
\newcommand{\omegab}{\boldsymbol{\omega}}
\newcommand{\taub}{\boldsymbol{\tau}}
\newcommand{\deltab}{\boldsymbol{\delta}}

\newcommand{\pb}{\mathbf{p}}
\newcommand{\vb}{\mathbf{v}}
\newcommand{\Xb}{\mathbf{X}}
\newcommand{\Yb}{\mathbf{Y}}
\newcommand{\qb}{\mathbf{q}}
\newcommand{\Rb}{\mathbf{R}}
\newcommand{\Mb}{\mathbf{M}}
\newcommand{\Wb}{\mathbf{W}}
\newcommand{\Jb}{\mathbf{J}}
\newcommand{\Ib}{\mathbf{I}}
\newcommand{\gb}{\mathbf{g}}

\newcommand{\Omegab}{\mathbf{\Omega}}
\newcommand{\Lambdab}{\mathbf{\Lambda}}

\newcommand{\Trm}{\textrm{T}}
\newcommand{\Tsf}{\textsf{T}}
\newcommand{\Gsf}{\textsf{G}}
\newcommand{\SO}{\textsf{SO}(3)}
\newcommand{\TSO}{\textsf{TSO}(3)}
\newcommand{\SOtwo}{\textsf{SO}(2)}
\newcommand{\SEthree}{\textsf{SE}(3)}
\newcommand{\SU}{\textsf{SU}(N)}

\newcommand{\so}{\mathfrak{so}}

\newcommand{\drm}{{\rm d}}
\newcommand{\Ad}{{\rm Ad}}
\newcommand{\ad}{{\rm ad}}

\newcommand{\be}{\begin{eqnarray}}
\newcommand{\ee}{\end{eqnarray}}

\newcommand{\nn}{\nonumber}

\newcommand{\beas}{\begin{eqnarray*}}
\newcommand{\eeas}{\end{eqnarray*}}
\newcommand{\la}{\label}

\begin{document}

\title{Geometric structure-preserving optimal control of the rigid body}

\author{Anthony M. Bloch}
\address{Alexander Ziwet Collegiate Professor of Mathematics and Department Chair, Mathematics, University of Michigan}
\email{abloch@umich.edu}
\author{Islam I. Hussein}
\address{Assistant Professor, Mechanical Engineering, Worcester Polytechnic Institute}
\email{ihussein@wpi.edu}
\author{Melvin Leok}
\address{Assistant Professor, Mathematics, Purdue University}
\email{mleok@math.purdue.edu}
\author{Amit K. Sanyal}
\address{Assistant Professor, Mechanical Engineering, University of Hawaii}
\email{aksanyal@hawaii.edu}

\date{Received: date / Revised: date}

\maketitle

\begin{abstract}

In this paper we study a discrete variational optimal control
problem for the rigid body. The cost to be minimized is the external
torque applied to move the rigid body from an initial condition to a
pre-specified terminal condition. Instead of discretizing the
equations of motion, we use the discrete equations obtained from the
discrete Lagrange--d'Alembert principle, a process that better
approximates the equations of motion. Within the discrete-time
setting, these two approaches are not equivalent in general. The
kinematics are discretized using a natural Lie-algebraic formulation
that guarantees that the flow remains on the Lie group $\SO$ and its
algebra $\so(3)$. We use Lagrange's method for constrained problems
in the calculus of variations to derive the discrete-time necessary
conditions. We give a numerical example for a three-dimensional
rigid body maneuver.

\end{abstract}

\section{Introduction}\label{sec:intro}

This paper deals with a structure-preserving computational approach
to the optimal control problem of minimizing the control effort
necessary to perform an attitude transfer from an initial state to a
prescribed final state, in the absence of a potential field. The
configuration of the rigid body is given by the rotation matrix from
the body frame to the spatial frame, which is an element of the
group of orientation-preserving isometries in $\mathbb{R}^3$. The
state of the rigid body is described by the rotation matrix and its
angular velocity.

To motivate the computational approach we adopt in the discrete-time
case, we first revisit the variational continuous-time optimal
control problem. The continuous-time extremal solutions to this
optimal control problem have certain special features, since they
arise from variational principles. General numerical integration
methods, including the popular Runge-Kutta schemes, typically
preserve neither first integrals nor the characteristics of the
configuration space. Geometric integrators are the class of
numerical integration schemes that preserve such properties, and a
good survey can be found in \cite{Hairer:02}.
Techniques particular to Hamiltonian systems are also discussed in
\cite{Leimkuhler:04} and \cite{SanzSerna:94}.

Our approach to discretizing the optimal control problem is in
contrast to traditional techniques such as collocation, wherein the
continuous equations of motion are imposed as constraints at a set
of collocation points. In our approach, modeled after
\cite{Junge:05}, the discrete equations of motion are derived from a
discrete variational principle, and this induces constraints on the
configuration at each discrete time step.

This approach yields discrete dynamics that are more faithful to the continuous
equations of motion, and consequently yields more accurate solutions to the
optimal control problem that is being approximated. This feature is extremely important in computing accurate (sub)optimal trajectories for long-term spacecraft attitude maneuvers. For example, in \cite{hussein:03c}, the authors propose an imaging spacecraft formation design that requires a continuous
attitude maneuver over a period of 77 days in a low Earth orbit. Hence,
attitude maneuver has to be very accurate to meet tight imaging constraints
over long time ranges.

While the discrete optimal control method presented here is illustrated using
the Lie group $\SO$ of rotation matrices, and its corresponding Lie algebra
$\so(3)$ of skew-symmetric matrices, we have derived the method with sufficient
generality to address the problem of optimal control on arbitrary Lie groups
with the drift vector field given by geodesic flow on the group, and it
therefore widely applicable.
For example, in inter-planetary orbit
transfers (see, for example, \cite{NewHorizons:06}), one is interested in
computing optimal or suboptimal trajectories on the group of rigid body motions
$\SEthree$ with a high degree of accuracy. Similar requirements also apply to
the control of quantum systems. For example, efficient construction of quantum
gates is a problem on the unitary Lie group $\SU$. This is an optimal control
problem, where one wishes to steer the identity operator to the  desired
unitary operator (see, for example, \cite{KhanejaGlaserBrockett:02} and
\cite{PalaoKosloff:02}).

Moreover, an important feature of the way we discretize the optimal
control problem is that it is $\SO$-equivariant. The
$\SO$-equivariance of our numerical method is desirable, since it
ensures that our results do not depend on the choice of coordinates
and coordinate frames. This is in contrast to methods based on
coordinatizing the rotation group using quaternions, (modified)
Rodrigues parameters, and Euler angles, as given in the survey
\cite{ScriTho94}. Even if the optimal cost function is
$\SO$-invariant, as in \cite{ScJuRo96}, the use of generalized
coordinates imposes constraints on the attitude kinematics.

For the purpose of numerical simulation, the corresponding discrete
optimal control problem is posed on the discrete state space as a
two stage discrete variational problem. In the first step, we derive
the discrete dynamics for the rigid body in the context of discrete
variational mechanics \cite{marsden:01}. This is achieved by
considering the discrete Lagrange--d'Alembert variational principle
\cite{Kane:00} in combination with essential ideas from Lie group
methods \cite{Iserles:00}, which yields a Lie group variational
integrator \cite{Leok:04}. This integrator explicitly preserves the
Lie group structure of the configuration space, and is similar to
the integrators introduced in \cite{Lee:05} for a rigid body in an
external field, and in \cite{Lee:05b} for full body dynamics. These
discrete equations are then imposed as constraints to be satisfied
by the extremal solutions to the discrete optimal control problem,
and we obtain the discrete extremal solutions in terms of the given
terminal states.

The paper is organized as follows. As motivation, in Section
\ref{sec:continuous_time}, we study the minimum control effort optimal control
problem in continuous-time. In Section \ref{sec:discrete_time}, we study the
corresponding discrete-time optimal control problem. In Section
\ref{subsec:formulation_discrete_example} we state the optimal control problem
and describe our approach. In Section \ref{sec:discrete_lag_RB}, we derive the
discrete-time equations of motion for the rigid body starting with the discrete
Lagrange--d'Alembert principle. These equations are used in Section
\ref{subsec:Discrete_Variational} to obtain the solution to the discrete
optimal control problem.  In Section \ref{sec:numerics}, we describe an
algorithm for solving the general nonlinear, implicit necessary conditions for
$\SO$ and give numerical examples for rest-to-rest and slew-up spacecraft maneuvers.

\section{Continuous-Time Results}\label{sec:continuous_time}

\subsection{Problem
Formulation}\label{subsec:formulation_continuous_example}

In this paper, the
natural pairing between $\so^*(3)$ and $\so(3)$ is denoted by
$\left<\cdot,\cdot\right>$. Let $\ll\cdot,\cdot\gg$ and $\ll\cdot,\cdot\gg_*$ denote
the standard (induced by the Killing form) inner product on $\so(3)$ and $\so^*(3)$,
respectively. The inner product $\ll\cdot,\cdot\gg_*$ is naturally
induced from the standard norm $\ll\xib,\omegab\gg=-\frac{1}{2}{\rm Tr}(\xib^T\omegab)$, for all $\xib,\omegab\in\so(3)$,
through\be\label{eq:inner_products}\ll\etab,\phib\gg_*&=&\left<\etab,\phib^{\sharp}\right>=\left<\etab,\omegab\right>=\left<\xib^{\flat},\omegab\right>\nn\\&=&\ll\xib,\omegab\gg,\ee where $\phib=\omegab^{\flat}\in\so^*(3)$ and
$\etab=\xib^{\flat}\in\so^*(3)$, with $\xib,\omegab\in\so(3)$ and $\flat$ and $\sharp$ are the musical isomorphisms with respect to the standard metric $\ll\cdot,\cdot\gg$. On $\so(3)$, these isomorphisms correspond to the transpose operation. That is, we have $\phib=\omegab^{\Trm}$ and
$\etab=\xib^{\Trm}$.

Let $\Jb:\so(3)\rightarrow\so^*(3)$ be the positive definite inertia operator. It can be shown that
\be\label{eq:inertia_property}\left<\Jb(\xib),\omegab\right>=\left<\Jb(\omegab),\xib\right>.\ee On $\so(3)$, $\Jb$ is given by $\Jb(\xib)=J\xib+\xib J$, where
$J$ is a positive definite symmetric matrix (see, for example, \cite{bloch:03,hussein:CDC05}). Moreover, we also have
$\Jb(\etab^{\sharp})^{\sharp}=(J\etab^{\Trm}+\etab^{\Trm}J)^{\Trm}=\Jb(\etab)$,
which is an abuse of notation since $\etab\in\so^*(3)$. For the sake of generality and mathematical precision
we will use the general definitions, though it helps to keep the above identifications
for $\so(3)$ in mind.

In this section we review some continuous-time optimal control
results using a simple optimal control example on $\SO$. The problem
we consider is that of minimizing the norm squared of the control
torque $\taub\in\so^*(3)$ applied to rotate a rigid body subject to the Lagrange--d'Alembert principle for the rigid
body\footnote{This is equivalent to constraining the problem to
satisfy the rigid body equations of motion given by
equations (\ref{eq:SO3_eoms_example}). However, for the sake
of generality that will be appreciated in the discrete-time problem,
we choose to treat the Lagrange--d'Alembert principle as the
constraint as opposed to the rigid body equations of motion. Both
are equivalent in the continuous-time case but are generally not equivalent in the discrete-time case.}
whose configuration is given by $\Rb\in\SO$ and body angular
velocity is given by $\Omegab\in\so(3)$. We require that the system
evolve from an initial state $(\Rb_0,\Omegab_0)$ to a final state
$(\Rb_T,\Omegab_T)$ at a fixed terminal time $T$.

Before proceeding with a statement of the optimal control problem, we first
define variations of the rigid body configuration $\Rb$ and its velocity
$\Omegab$. Let $\Wb(t)\in\so(3)$ be the variation vector field associated with
a curve $\Rb(t)$ on $\SO$ \cite{agrachev:04,Warner:73}. The vector field
$\Wb(t)$ satisfies \be
\deltab\Rb(t)
  = \Rb\Wb\in\Tsf_{\Rb(t)}\SO, ~\Wb(0)=\mathbf{0}, ~\Wb(T)=\mathbf{0},\nn
\ee where $\deltab\Rb$ is defined by $\deltab\Rb(t) =
  \partial\Rb_{\epsilon}(t)/\partial\epsilon\big|_{\epsilon=0}$, with $\Rb_{\epsilon}(t):=\Rb(t,\epsilon)$ is the variation of
the curve $\Rb(t)$ that satisfies $\Rb(t,0)=\Rb(t)$. The variation in the
velocity vector field is denoted $\deltab\Omegab$. For a deeper understanding
of variations of general vector fields, see for example the treatment in
\cite{hussein:04CDCb}.

We now state the minimum control effort optimal control problem.

\begin{prb}\label{pb:SO3_example0} Minimize
\be\label{eq:SO3_cost_LD_example0}
  \Jcal &=&
  \frac{1}{2}\int_0^T\ll\taub,\taub\gg_*\drm t
\ee subject to
\begin{enumerate}
    \item satisfying \emph{Lagrange--d'Alembert principle}:
     \be\label{eq:SO3_LD_example0}
      \deltab\int_0^T\frac{1}{2}\left<\Jb\left(\Omegab\right),\Omegab\right>\drm t+\int_0^T\left<\taub,\Wb\right>\drm t=0,
    \ee
    for a variation vector field $\Wb(t)$, and subject to $\dot{\Rb}=\Rb\Omegab$,
    \item and the \emph{boundary conditions} \be\label{eq:SO3_bcs_example0}
    \Rb(0)&=&\Rb_0, ~\Omegab(0)=\Omegab_0,\nn\\ \Rb(T)&=&\Rb_T,
    ~\Omegab(T)=\Omegab_T.\ee
\end{enumerate}
\end{prb}

We now show that the constraint of satisfying the Lagrange--d'Alembert principle leads to the following problem formulation,
where the rigid body equations of motion replace the Lagrange--d'Alembert principle.

\begin{prb}\label{pb:SO3_example} Minimize
\be\label{eq:SO3_cost_example}
  \Jcal &=&
  \frac{1}{2}\int_0^T\ll\taub,\taub\gg_*\drm t
\ee subject to
\begin{enumerate}
    \item the \emph{dynamics} \be\label{eq:SO3_eoms_example}
      \dot{\Rb} &=& \Rb\Omegab\\
      \dot{\Mb} &=&\ad_{\Omegab}^*\Mb+\taub=
      \left[\Mb,\Omegab\right]+\taub,\nn\ee
      where $\Mb=\Jb(\Omegab)\in\so^*(3)$ is the momentum,
    \item and the \emph{boundary conditions} \be\label{eq:SO3_bcs_example}
    \Rb(0)&=&\Rb_0, ~\Omegab(0)=\Omegab_0,\nn\\ \Rb(T)&=&\Rb_T, ~\Omegab(T)=\Omegab_T.
\ee
\end{enumerate}
\end{prb}

In the above, $\ad^*$ is the dual of the adjoint representation,
$\ad$, of $\so(3)$ and is given by
$\ad_{\xib}^*\etab=-[\xib,\etab]\in\so^*(3)$, for all $\xib\in\so(3)$ and
$\etab\in\so^*(3)$. Recall that the bracket is defined by
$\left[\xib,\omegab\right]=\xib\omegab-\omegab\xib$.

\subsection{The Lagrange--d'Alembert Principle and the Rigid
Body Equations of Motion}\label{subsec:LD_continuous_example}

In this section we derive the forced rigid body equations of motion
(equations (\ref{eq:SO3_eoms_example})) from the Lagrange--d'Alembert
principle. We begin by appending the constraint
$\dot{\Rb}=\Rb\Omegab$ to the Lagrangian \be
    0&=&\deltab\int_0^T\left(\frac{1}{2}\left<\Jb\left(\Omegab\right),\Omegab\right>+\left<\Lambdab,\Rb^{-1}\dot{\Rb}-\Omegab\right>\right)\drm t\nn\\&&
    +\int_0^T\left<\taub,\Wb\right>\drm t,\nn
\ee where $\Lambdab\in\so^*(3)$ is a Lagrange multiplier. Taking variations, we
obtain \be\label{eq:SO3_cost_LD_variations}
  0&=& \int_0^T
  \bigg(\left<-\dot{\Lambdab}-\left[\Omegab,\Lambdab\right]+\taub,\Wb\right> +\left<-\Lambdab+\Jb\left(\Omegab\right),\deltab\Omegab\right>\bigg)\drm
   t\nn\\&&+\left[\left<\Lambdab,\Wb(t)\right>\right]_0^T.\nn
\ee

The term outside the integral vanishes by virtue of the boundary conditions on
$\Wb(t)$. Since $\deltab\Omegab$ and $\Wb$ are arbitrary and independent, we
must have $\big<-\Lambdab+\Jb\left(\Omegab\right),\deltab\Omegab\big>=0$ and
hence $\Lambdab=\Jb\left(\Omegab\right)$ which is the body angular momentum.
For the remainder of this section, we set
$\Mb=\Lambdab=\Jb\left(\Omegab\right)$. The first term in the integrand gives
us the forced second order dynamics of the rigid body which is \be
\dot{\Mb}=\left[\Mb,\Omegab\right]+\taub.\nn\ee This completes the proof that
the Problem (\ref{pb:SO3_example0}) is equivalent to Problem
(\ref{pb:SO3_example}).

\emph{\textbf{A Direct Approach.}} We now give a direct derivation that does
not involve the use of Lagrange multipliers. This approach can be found in
Section 13.5 in \cite{marsden:99}. First, we take variations of the kinematic
condition $\Omegab=\Rb^{-1}\dot\Rb$ to obtain $\deltab\Omegab =
  -\Rb^{-1}\left(\deltab\Rb\right)\Rb^{-1}\dot{\Rb}+\Rb^{-1}\left(\deltab\dot{\Rb}\right)$. As defined previously, we have $\Wb=\Rb^{-1}\deltab\Rb$ and,
therefore, $\dot{\Wb}=-\Rb^{-1}\dot{\Rb}\Rb^{-1}\deltab\Rb+\Rb^{-1}
\deltab\dot\Rb=-\Omegab\Wb+\Rb^{-1}\deltab\dot\Rb$, since
$\deltab\dot{\Rb}=\frac{\drm}{\drm t}\deltab\Rb$ (see for example
\cite{milnor:63}, page 52). Hence, we have
\be\label{eq:direct_variation_cont_example}
  \deltab\Omegab =
  -\Wb\Omegab+\Omegab\Wb+\dot{\Wb}
  =\ad_{\Omega}\Wb+\dot{\Wb}.
\ee Taking variations of the Lagrange--d'Alembert principle we obtain \be
    \int_0^T\left<\Jb\left(\Omegab\right),\deltab\Omegab\right>+\left<\taub,\Wb\right>\drm t=0.\nn
\ee Using the variation in equation
(\ref{eq:direct_variation_cont_example}) and integrating by parts,
we obtain \be
    0=\int_0^T\left<-\dot{\Mb}+\ad_{\Omega}^*\Mb+\taub,\Wb\right>\drm
    t+\left[\left<\Jb\left(\Omegab\right),\Wb(t)\right>\right]_0^T,\nn
\ee where we have used the identity
\be\label{eq:ad_property}\left<\etab,\ad_{\omegab}\xib\right>=\left<\ad_{\omegab}^*\etab,\xib\right>,
~\etab\in\so^*(3), ~\omegab,\xib\in\so(3).\ee This gives the desired result,
with $\Mb=\Jb\left(\Omegab\right)$.

In Section \ref{subsec:Continuous_Variational}, we demonstrate how
the necessary conditions for Problem (\ref{pb:SO3_example}) are
derived using a variational approach.

\subsection{Continuous-Time Variational Optimal Control Problem}\label{subsec:Continuous_Variational}

A direct variational approach is used here to obtain the
differential equation that satisfies the optimal control Problem
(\ref{pb:SO3_example}).

\emph{\textbf{A Second Order Direct Approach.}} ``Second order" is
used here to reflect the fact that we now study variations of second
order dynamical equations as opposed to the kinematic direct
approach studied in Section \ref{subsec:LD_continuous_example}. We
now give the resulting necessary conditions using a direct approach
as in \cite{marsden:99}. We already computed the variations of $\Rb$
and $\Omegab$. These were as follows: $\deltab\Rb=\Rb\Wb$ and
$\deltab\Omegab=\ad_{\Omegab}\Wb+\dot{\Wb}$. We now compute the
variation of $\dot{\Mb}$ with the goal of obtaining the proper
variations for $\taub$: \be
  \deltab\dot{\Mb} &=&
  \Jb\left(\deltab\dot{\Omegab}\right)=\Jb\left(\frac{\drm}{\drm
  t}\deltab\Omegab+\Rcal\left(\Wb,\Omegab\right)\Omegab\right),\nn
\ee where $\Rcal$ is the curvature tensor on SO(3). The curvature
tensor $\Rcal$ arises due to the identity (see \cite{milnor:63},
page 52) \be\label{eq:curvature}
    \frac{\partial}{\partial\epsilon}\frac{\partial}{\partial t}\mathbf{Y}-
    \frac{\partial}{\partial t}\frac{\partial}{\partial
    \epsilon}\mathbf{Y}=\Rcal(\Wb,\mathbf{Y})\Omegab,\nn
\ee where $\mathbf{Y}\in\TSO$ is any vector field along the curve
$\Rb(t)\in\SO$. Taking variations of
$\dot{\Mb}=\ad_{\Omegab}^*\Mb+\taub$, we obtain $\deltab\dot{\Mb}
  =\ad_{\deltab\Omegab}^*\Mb+\ad_{\Omegab}^*\deltab\Mb+
  \deltab\taub$. We now have the desired variation in $\taub$:
\be\la{eq:tau_variations}
  \deltab\taub &=&
  \Jb\left(\Rcal\left(\Wb,\Omegab\right)\Omegab\right)+\frac{\drm }{\drm
  t}\Jb\left(\deltab\Omegab\right)-\ad_{\deltab\Omegab}^*\Mb-\ad_{\Omegab}^*\deltab\Mb.
\ee Taking variations of the cost functional (\ref{eq:SO3_cost_example}) we
obtain: \be
  &&\deltab\Jcal =   \int_0^T\bigg(\big<\Jb(\ddot{\varsigma})-\ad_{\Omegab}^*\left(\Jb(\dot{\varsigma})\right)+\dot{\etab}-\frac{\drm}{\drm
  t}\left(\ad_{\varsigmab}^*\Mb\right)\nn\\
  &&+\left[\Rcal\left(\Jb(\varsigmab)^{\sharp},\Omegab\right)\Omegab\right]^{\flat}+\ad_{\Omegab}^*\ad_{\varsigmab}^*\Mb
  -\ad_{\Omegab}^*\etab,\Wb\big>\bigg)\drm
  t,\nn
\ee where $\varsigmab=\taub^{\sharp}\in\so(3)$ and
$\etab=\Jb\left(\ad_{\Omegab}\varsigmab\right)\in\so^*(3)$. In obtaining the
above expression, we have used integration by parts and the boundary conditions
(\ref{eq:SO3_bcs_example}), equations (\ref{eq:direct_variation_cont_example})
and (\ref{eq:tau_variations}), and the identities (\ref{eq:inner_products}),
(\ref{eq:inertia_property}) and (\ref{eq:ad_property}). Hence, we have the
following theorem.

\begin{thm}\label{thm:SO3_example2} The necessary optimality conditions for the problem of minimizing
(\ref{eq:SO3_cost_example}) subject to the dynamics
(\ref{eq:SO3_eoms_example}) and the boundary conditions
(\ref{eq:SO3_bcs_example}) are given by the single fourth
order\footnote{Second order in $\taub$ and fourth order in $\Rb$.}
differential equation \be
  0 &=& \Jb(\ddot{\varsigma})-\ad_{\Omegab}^*\left(\Jb(\dot{\varsigma})\right)+\dot{\etab}-\frac{\drm}{\drm
  t}\left(\ad_{\varsigmab}^*\Mb\right)\nn\\
  &&+\left[\Rcal\left(\left(\Jb(\varsigmab)\right)^{\sharp},\Omegab\right)\Omegab\right]^{\flat}+\ad_{\Omegab}^*\left(\ad_{\varsigmab}^*\Mb\right)
  -\ad_{\Omegab}^*\etab,\nn
\ee as well as the equations (\ref{eq:SO3_eoms_example}) and the
boundary conditions (\ref{eq:SO3_bcs_example}), where $\varsigmab$
and $\etab$ are as defined above.
\end{thm}

Note that for a compact semi-simple Lie group $\Gsf$ with Lie algebra
$\mathfrak{g}$, the curvature tensor, with respect to a bi-invariant metric, is
given by (see \cite{milnor:63}): \be\label{eq:curvature_compact}
\Rcal\left(\mathbf{X},\mathbf{Y}\right)\mathbf{Z}=\frac{1}{4}\ad_{\ad_{\Xb}\Yb}\mathbf{Z},
\ee for all $\mathbf{X},\mathbf{Y},\mathbf{Z}\in\mathfrak{g}$.

Using a Lagrange multiplier approach as in Section
\ref{subsec:LD_continuous_example}, we may show that the result of
Theorem \ref{thm:SO3_example2} is equivalent to the following
theorem.

\begin{thm}\label{thm:SO3_example} The necessary optimality conditions
for the problem of minimizing (\ref{eq:SO3_cost_example}) subject to
the dynamics (\ref{eq:SO3_eoms_example}) and the boundary conditions
(\ref{eq:SO3_bcs_example}) are given by
\be\label{eq:SO3_theorem_example}
  \taub &=& \Lambdab_2\nn \\
\dot{\Lambdab}_1&=&\left[\Rcal\left(\Jb\left(\Lambdab_2\right)^{\sharp},\Omegab\right)\Omegab\right]^{\flat}+\ad_{\Omegab}^*\Lambdab_1\\
\dot{\Lambdab}_2&=&-\Jb^{-1}\left(\Lambdab_1\right)-\ad_{\Omegab}\Lambdab_2+\Jb^{-1}\left(\ad_{\Lambdab_2}^*\Mb\right)\nn
\ee as well as the equations (\ref{eq:SO3_eoms_example}) and the
boundary conditions (\ref{eq:SO3_bcs_example}), where the Lagrange multipliers
$\Lambdab_1\in\so^*(3), ~\Lambdab_2\in\so(3)$ correspond to the kinematic and dynamics constraints (\ref{eq:SO3_eoms_example}), respectively.
\end{thm}

\begin{rem} Note that the equations of motion that
arise from the Lagrange--d'Alembert principle are used to define the
dynamic constraints. So, in effect, we are minimizing $\Jcal$
subject to satisfying the Lagrange--d'Alembert principle for the
rigid body. Analogously, the discrete version of the
Lagrange--d'Alembert principle will be used to derive the discrete
equations of motion in the discrete optimal control problem to be
studied in Section \ref{subsec:Discrete_Variational}. This view is
in line with the approach in \cite{Junge:05} in that we do not
discretize the equations of motion directly, but, instead, we
discretize the Lagrange--d'Alembert principle. These two approaches
are not equivalent in general.\end{rem}

\begin{cor} The necessary optimality conditions of Theorem \ref{thm:SO3_example2} are equivalent to the necessary conditions of
Theorem \ref{thm:SO3_example}.\end{cor}

\begin{proof} In Theorem \ref{thm:SO3_example}, differentiate $\Lambdab_2$ once and then use all three
differential equations to replace $\Lambdab_1$ and $\Lambdab_2$ with
expressions involving only $\taub, ~\Mb$ and $\Omegab$.\end{proof}

\section{Discrete-Time Results}\label{sec:discrete_time}

\subsection{Problem Formulation}\label{subsec:formulation_discrete_example}

In this section we give the discrete version of the problem
introduced in Section \ref{subsec:formulation_continuous_example}.
So, we consider minimizing the norm squared of the control torque
$\taub$ subject to satisfaction of the discrete Lagrange--d'Alembert
principle for the rigid body whose configuration and body angular
velocity at time step $t_k$ are given by $\Rb_k\in\SO$ and
$\Omegab_k\in\so(3)$, respectively. The kinematic constraint may be
expressed as \be\label{eq:discrete_RB_kinematics}
  \Rb_{k+1} = \Rb_k\exp\left(h\Omegab_k\right)=\Rb_k\gb_k,
\ee where $h$ is the integration time step, $\exp:\so(3)\rightarrow\SO$ is the
exponential map and $\gb_k=\exp(h\Omegab_k)$. The boundary conditions are given
by $(\Rb^*_0,\Omegab^*_0)$ and $(\Rb^*_N,\Omegab^*_{N-1})$, where $t_0=0$ and
$N=T/h$ is such that $t_N=T$.

More generally, one considers the ansatz $\Rb_{k+1} =
\Rb_k\exp\left(\Omegab(h)\right)$, where $\Omegab(\cdot)$ is an interpolatory
curve in $\so(3)$ parameterized by the angular velocity at internal nodal
points. This allows one to construct Lie group variational integrators of
arbitrarily high order \cite{Leok:04}. To simplify the subsequent treatment, we
adopt (\ref{eq:discrete_RB_kinematics}) as the kinematic constraint, which
yields a first-order accurate Lie symplectic Euler method, which will
nevertheless have effective order two as it is symplectically conjugate to the
second-order accurate Lie St\"ormer--Verlet method (see,
\S\ref{sec:symplectic_equivalence}).

The reason we constrain $\Omegab$ at $t=h(N-1)$ instead of at $t=hN$
will become clear when we derive the discrete equations of motion in
Section \ref{sec:discrete_lag_RB}. A simple explanation for this is
that a constraint on $\Omegab_k\in\so(3)$ corresponds, by left
translations to a constraint on $\dot{\Rb}_k\in\Tsf_{\Rb_k}\SO$. In
turn, in the discrete setting and depending on the choice of
discretization, this corresponds to a constraint on the neighboring
discrete points
$\ldots,\Rb_{k-2},\Rb_{k-1},\Rb_{k+1},\Rb_{k+2},\ldots$. With our
choice of discretization (equation
(\ref{eq:discrete_RB_kinematics})), this corresponds to constraints
on $\Rb_k$ and $\Rb_{k+1}$. Hence, to ensure that the effect of the
terminal constraint on $\Omegab$ is correctly accounted for, the
constraint must be imposed on $\Omegab_{N-1}$, which entails some
constraints on variations at both $\Rb_{N-1}$ and
$\Rb_N$. We will return to this point later in the paper.

The discrete kinematic constraint ensures that the sequence $\Rb_k$ stays on
the rotation group, since the exponential of the angular velocity matrix
$\Omegab_k$, which is in the algebra $\so(3)$, is a rotation matrix, and the
rotation group is closed under matrix multiplication. This is natural to do in
the context of discrete variational numerical solvers (for both initial value
and two point boundary value problems).

Following the methodology of \cite{Junge:05}, we have the following optimal
control problem.

\begin{prb}\label{pb:SO3_discrete_example0} Minimize
\be\label{eq:SO3_cost_LD_discrete_example0}
  \Jcal &=&
  \sum_{k=0}^N\frac{1}{2}\ll\taub_k,\taub_k\gg_*
\ee subject to
\begin{enumerate}
    \item satisfying the \emph{discrete Lagrange--d'Alembert principle}:
\be\label{eq:SO3_LD_discrete_example0}
      \deltab\sum_{k=0}^{N-1}\frac{1}{2}\left<\Jb\left(\Omegab_k\right),\Omegab_k\right>+\sum_{k=0}^N\left<\taub_k,\Wb_k\right>=0,
    \ee
    subject to $\Rb_0=\Rb_0^*$, $\Rb_N=\Rb_N^*$ and $\Rb_{k+1} =\Rb_k\gb_k$,
    $k=0,1,\ldots,N-1$, where $\Wb_k$ is the variation vector field at time step $t_k$ satisfying
    $\deltab\Rb_k=\Rb_k\Wb_k$,
    \item and the \emph{boundary conditions} \be\label{eq:SO3_bcs_discrete_example0}
    \Rb_0&=&\Rb_0^*, ~\Omegab_0=\Omegab^*_0, \nn\\
    \Rb_N&=&\Rb_N^*, ~\Omegab_{N-1}=\Omegab^*_{N-1}.
\ee
\end{enumerate}
\end{prb}

In Problem \ref{pb:SO3_discrete_example0}, the discrete Lagrange--d'Alembert principle is used to derive the equations of motion for
the rigid body with initial and terminal configuration constraints.
Hence, we get a two point boundary value problem. The full
configuration and velocity boundary conditions come into the picture
when we study the optimal control problem. We will show that the
constraint of satisfying the Lagrange--d'Alembert principle in
Problem \ref{pb:SO3_discrete_example0} leads to the following
problem formulation, where the discrete rigid body equations of
motion replace the Lagrange--d'Alembert principle constraint. Only
when addressing the following optimal control problem will we need to include the velocity boundary conditions
in the derivation.

\begin{prb}\label{pb:SO3_discrete_example} Minimize
\be\label{eq:SO3_cost_LD_discrete_example}
  \Jcal &=&
  \sum_{k=0}^N\frac{1}{2}\ll\taub_k,\taub_k\gg_*
\ee subject to
\begin{enumerate}
    \item the \emph{discrete dynamics} \begin{align}\label{eq:SO3_eoms_discrete_example}
      &\Rb_{k+1} = \Rb_k\gb_k, ~k=0,\ldots,N-1\nn\\
      &\Mb_k =
      \Ad_{\gb_k}^*\left(h\taub_k+\Mb_{k-1}\right), ~k=1,\ldots,N-1,\nn\\
      &\Mb_k=\Jb\left(\Omegab_k\right), ~k=0,\ldots,N-1,
    \end{align}
    \item and the \emph{boundary conditions} \vspace*{-3mm}
    \be\label{eq:SO3_bcs_discrete_example}
    \Rb_0&=&\Rb^*_0, ~\Omegab_0=\Omegab^*_0, \nn\\
    \Rb_N&=&\Rb^*_N,
    ~\Omegab_{N-1}=\Omegab^*_{N-1}.
\ee
\end{enumerate}
\end{prb}

Regarding terminal velocity conditions, note that
in the second of equations (\ref{eq:SO3_eoms_discrete_example}) if
we let $k=N$ we find that $\Omegab_N$ appears in the equation. A constraint on $\Omegab_N$ dictates
constraints at the points $\Rb_N$ and $\Rb_{N+1}$ through the first
equation in (\ref{eq:SO3_eoms_discrete_example}). Since we only
consider time points up to $t=Nh$, we can not allow $k=N$ in the
second of equations (\ref{eq:SO3_eoms_discrete_example}) and hence
our terminal velocity constraints are posed in terms of
$\Omegab_{N-1}$ instead of $\Omegab_N$.

As mentioned above, $\Wb_k$ is a variation vector field associated with the
perturbed group element $\Rb_k^{\epsilon}$. Likewise, we need to define a
variation vector field associated with the element $\gb_k=\exp(h\Omegab_k)$.
First, let the perturbed variable $\gb_k^{\epsilon}$ be defined by
\be\label{eq:variations_in_gk} \gb_k^{\epsilon} = \gb_k \exp(\epsilon h \deltab
\Omegab_k),\ee where
\be\deltab\Omegab_k=\frac{\partial\Omegab_k^{\epsilon}}{\partial
\epsilon}\bigg|_{\epsilon=0}.\nn\ee  Note that
$\gb_k^{\epsilon}\big|_{\epsilon=0}=\gb_k$ as desired. Moreover, we
have \be\label{eq:Deltab_g} \deltab \gb_k
  = \gb_k (h\deltab\Omegab_k) \exp(\epsilon h \delta \Omegab_k)\big|_{\epsilon=0}= h\gb_k
  \deltab\Omegab_k.\ee
This will be needed later when taking variations.

\subsection{The Discrete Lagrange--d'Alembert Principle and the Rigid
Body Equations of Motion}\label{sec:discrete_lag_RB}

In this section we derive the discrete forced rigid body equations
of motion (equations (\ref{eq:SO3_eoms_discrete_example})) starting
with the discrete Lagrange--d'Alembert principle. We begin by
rewriting the kinematic constraint as
$\exp^{-1}\left(\Rb^{-1}_k\Rb_{k+1}\right)=h\Omegab_k$, which is an
expression on the Lie algebra $\so(3)$, and appending it to the
Lagrangian \be
  0&=&\sum_{k=0}^N\left<\taub_k,\Wb_k\right>+\deltab\sum_{k=0}^{N-1}\bigg(\frac{1}{2}\left<\Jb\left(\Omegab_k\right),\Omegab_k\right>\nn\\
  &&+\left<\Mb_k,\frac1{h}\exp^{-1}\left(\Rb^{-1}_k\Rb_{k+1}\right)-\Omegab_k\right>\bigg),\nn
\ee where $\Mb_k\in\so^*(3)$, $k=0,1,\ldots,N-1$, are Lagrange multipliers.
These multipliers enforce the kinematic discrete equations. We could have added
additional terms to enforce the configuration boundary conditions, allowing for
$\Wb_0$ and $\Wb_N$ to be arbitrary and non-zero. Instead, we elect to enforce
the constraints by requiring $\Wb_0=0$ and $\deltab\Omegab_0=0$. These two
approaches are of course equivalent.

Taking variations of $\frac1{h}\exp^{-1}\left(\Rb^{-1}_k\Rb_{k+1}\right)$
$-\Omegab_k=0$ is equivalent to taking variations of the original expression
$\Rb^{-1}_k\Rb_{k+1} =\exp\left(h\Omegab_k\right)$, and is easier to compute
since it as an expression over the Lie algebra. Once the variations are
computed, one can easily obtain the Lie algebra-equivalent of the variations as
follows. First take variations of the kinematics
(\ref{eq:discrete_RB_kinematics}) to get
$-\Rb_k^{-1}\left(\deltab\Rb_k\right)\Rb^{-1}_k\Rb_{k+1}+\Rb^{-1}_k\deltab\Rb_{k+1}=h\gb_k\cdot\deltab\Omegab_k$,
which is equivalent to $-\Wb_k\gb_k+\gb_k\Wb_{k+1}=h\gb_k\deltab\Omegab_k$, or
\be\label{eq:SO3_variations_example_discrete}
\deltab\Omegab_k=\frac{1}{h}\left[-\Ad_{\gb_k^{-1}}\Wb_k+\Wb_{k+1}\right]. \ee
Note that this is an expression over the Lie algebra $\so(3)$.

After simple algebraic and re-indexing operations, the Lagrange--d'Alembert principle gives \begin{align}
  &0=\left<\taub_0-\frac{1}{h}\Ad_{\gb_0^{-1}}^*\Mb_{0},\Wb_0\right>
  +\left<\taub_N+\frac{1}{h}\Mb_{N-1},\Wb_N\right>\nn\\
  &+\sum_{k=0}^{N-1}\left<\Jb\left(\Omegab_k\right)-\Mb_k,\deltab\Omegab_{k}\right>\nn\\&
  +\sum_{k=1}^{N-1}\left<\taub_k
  -\frac{1}{h}\Ad_{\gb_k^{-1}}^*\Mb_{k}+\frac{1}{h}\Mb_{k-1},\Wb_k\right>.\nn
\end{align} By the boundary conditions $\Rb_0=\Rb_0^*$ and $\Rb_N=\Rb_N^*$,
we have $\Wb_0=0$ and $\Wb_N=0$. Since $\deltab\Omegab_k$, $k=0,\ldots,N-1$,
and $\Wb_k$, $k=1,\ldots,N-1$, are arbitrary and independent, then the
Lagrange--d'Alembert principle requires that the equations
(\ref{eq:SO3_eoms_discrete_example}) hold true. The variables $\Mb_k$,
$k=0,\ldots,N-1$, are of course nothing but the discrete angular momentum of
the rigid body.
 The equations
(\ref{eq:SO3_eoms_discrete_example}) can be viewed in two ways. The first is to
consider the two point boundary value problem where we retain the terminal
condition on $\Rb_N$. In this case a (constrained) variety of a combination of
control torques $\taub_k$, $k=0,\ldots,N$, and initial velocity conditions
$\Omegab_0$ can be chosen to drive the rigid body from the initial condition
$\Rb_0$ to the terminal condition $\Rb_N$. The second view is to treat it as an
initial value problem by ignoring any terminal configuration constraints. In
this case $\Wb_N\ne 0$ and any combination of control torques $\taub_k$,
$k=0,\ldots,N$, and initial velocity conditions $\Omegab_0$ can be chosen
freely.

\emph{\textbf{A Direct Approach.}} Taking direct variations of the
cost functional we obtain \begin{align}
  &0=
  \left<\taub_0-\frac{1}{h}\Ad_{\gb_0^{-1}}^*\Jb\left(\Omegab_{0}\right),\Wb_0\right>\nn\\
  &+\left<\taub_N+\frac{1}{h}\Jb\left(\Omega_{N-1}\right),\Wb_N\right>\nn\\
  &+\sum_{k=1}^{N-1}\left<\taub_k
  -\frac{1}{h}\Ad_{\gb_k^{-1}}^*\Jb\left(\Omegab_{k}\right)+\frac{1}{h}\Jb\left(\Omegab_{k-1}\right),\Wb_k\right>.\nn
\end{align} where we have used equation
(\ref{eq:SO3_variations_example_discrete}). This gives the same
equations of motion as those in equation
(\ref{eq:SO3_eoms_discrete_example}).

\paragraph{Simulation Results.} To test our results, we re-write the discrete equations
(\ref{eq:SO3_eoms_discrete_example}) for the subgroup $\SOtwo$. For
$\SOtwo$ we have
\be\label{eq:R_SO2}\Rb_k=\left[\begin{array}{cc}
\cos\theta_k & -\sin\theta_k \\
\sin\theta_k &\cos\theta_k
\end{array}\right], ~\Omegab_k=\left[\begin{array}{cc}
0 & -\omega_k \\
\omega_k & 0
\end{array}\right]\ee
and
\be\label{eq:exp_SO2}\exp\left(\Omegab_k\right)=\left[\begin{array}{cc}
\cos\omega_k & -\sin\omega_k \\
\sin\omega_k & \cos\omega_k
\end{array}\right].\ee
The inertia operation is simply given
by\be\label{eq:J_SO2}\Jb\left(\Omegab_k\right)=\left[\begin{array}{cc}
0 & -I\omega_k \\
I\omega_k & 0
\end{array}\right],\ee
where $I$ is the mass moment of inertia of the body about the
out-of-plane axis. One can check that
$\Ad_{\exp(\omegab)}\xib=\xib$ and that
$\Ad_{\exp(\omegab)}^*\etab=\etab$, for all $\xib,
~\omegab\in\so(2)$ and $\etab\in\so^*(2)$.

Then the equations (\ref{eq:SO3_eoms_discrete_example}) (treated as
an initial value problem) are given for $\SOtwo$ by
\be\label{eq:eoms_disc_so2}\theta_{k+1}&=&\theta_k+h\omega_k,
~k=0,\ldots,N-1\nn\\ \omega_k&=&\frac{h}{I}\tau_k+\omega_{k-1},
~k=1,\ldots,N-1\ee in addition to the initial
conditions $\theta_0=\theta_0^*, ~\omega_0=\omega_0^*$.

To verify the accuracy of our numerical computation, we give the corresponding
continuous-time equations of motion for the planar rigid body on $\SOtwo$ using
equations (\ref{eq:SO3_eoms_example}). The Lie bracket on $\SOtwo$ is
identically equal to zero. Hence, one can check that the equations
(\ref{eq:SO3_eoms_example}) are given by $\dot{\theta}=\omega$,
$\dot{\omega}=\frac{\tau}{I}$, where $\theta$, $\omega$ and $\tau$ are the
continuous time angular position, velocity and torque, respectively. We
integrate the equations using the torque $\tau(t)=\sin\left(\frac{\pi
t}{2}\right)$, $t\in[0,T]$. We use the following parameters for our
simulations: $T=10$, $I=1$, $\theta(0)=3$, $\omega(0)=4$ and we try three
different time steps corresponding to $N=1000,1500,2000$. The error between the
continuous- and discrete-time values of $\theta$ and $\omega$ are given in
Figure (\ref{fig:result_eoms_so2}). Note that the accuracy of the simulation
improves with increasing $N$.

\begin{figure}[!h]
\center{\includegraphics*[width = 3.7in]{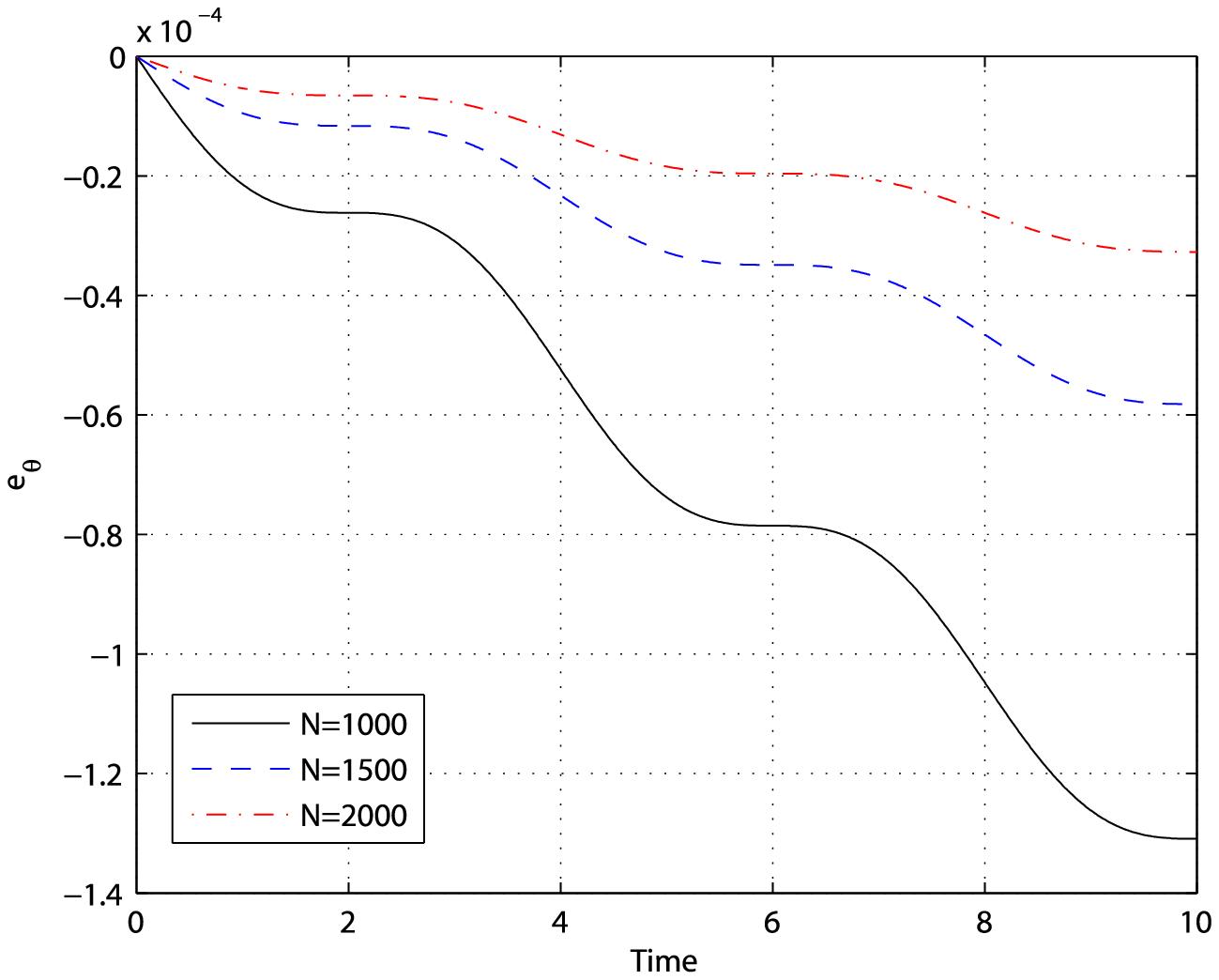}\\\includegraphics*[width
= 3.7in]{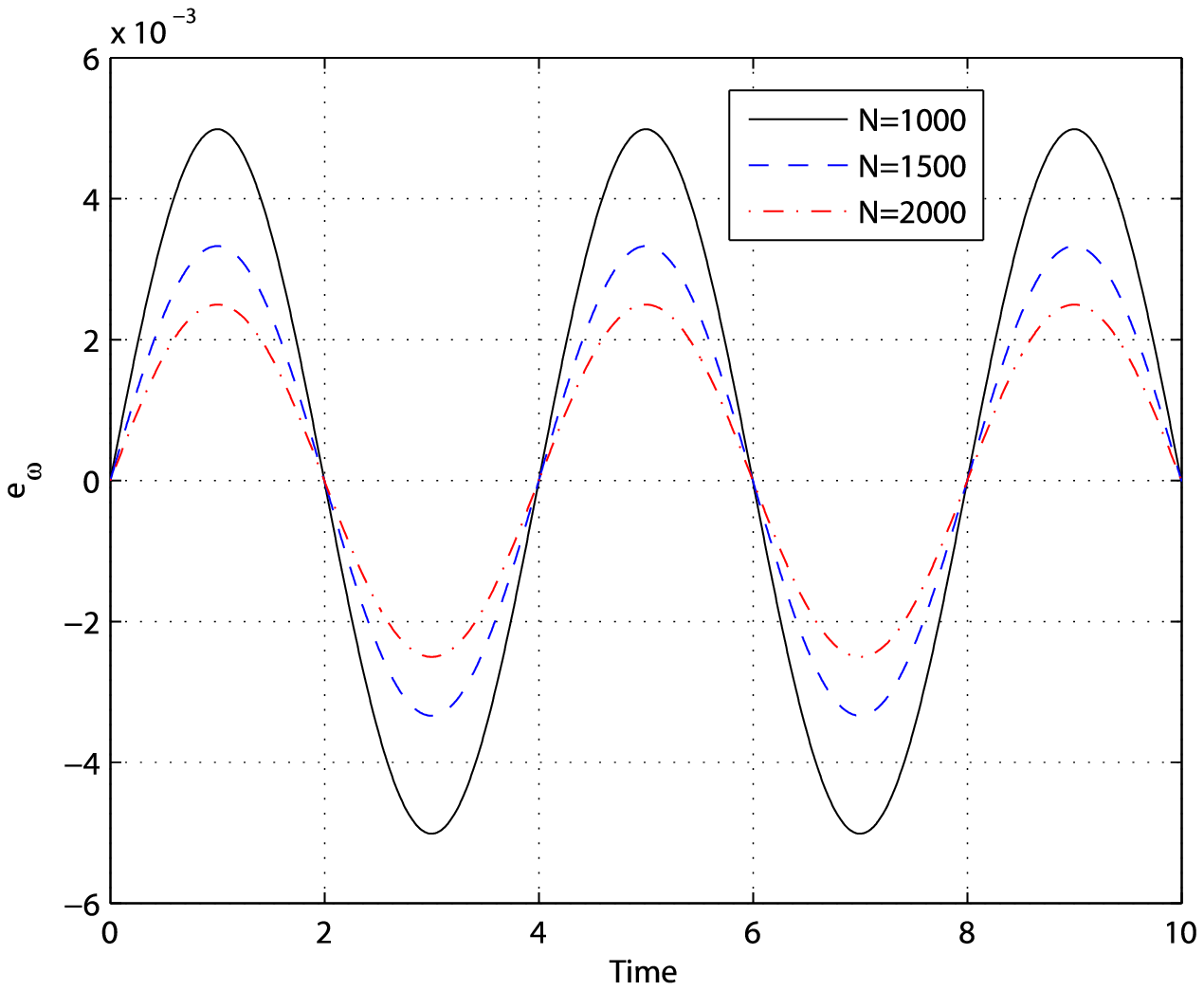}}\caption{Error dynamics on
$\SOtwo$.}\label{fig:result_eoms_so2}
\end{figure}

\begin{rem} Note that the discrete-time equations
(\ref{eq:eoms_disc_so2}) correspond to the Euler approximation for
the equations of motion. This is a check that our method
returns something familiar for a simple example as the planar rigid
body. However, we emphasize that on $\SO$ the discretization will
not necessarily be equivalent to any of the classical discretization
schemes. The discretization will generally result in a set of
nonlinear implicit algebraic equations.\end{rem}

\subsection{Discrete-Time Variational Optimal Control Problem}\label{subsec:Discrete_Variational}

We now address Problem \ref{pb:SO3_discrete_example} by first
forming the appended cost functional: \be
  \Jcal &=&
  \sum_{k=0}^N\frac{1}{2}\ll\taub_k,\taub_k\gg_*\nn\\
  &&+\sum_{k=0}^{N-1}\left<\Lambdab_k^1,-h\Omegab_k+\exp^{-1}\left(\Rb^{-1}_k\Rb_{k+1}\right)\right>\nn\\
  &&
  +\sum_{k=1}^{N-1}\left<\Mb_k-\Ad_{\gb_k}^*\left(h\taub_k+\Mb_{k-1}\right),\Lambdab_k^2\right>.\nn
\ee Writing $\Mb_k=\Ad_{\gb_k}^*\left(h\taub_k+\Mb_{k-1}\right)$ as
$\Mb_k=\gb_k^{-1}\left(h\taub_k+\Mb_{k-1}\right)\gb_k$ and taking
variations of this expression, we obtain \be
\deltab\Mb_k&=&-h\deltab\Omegab_k\gb_k^{-1}\left(h\taub_k+\Mb_{k-1}\right)\gb_k
\nn\\&&+\gb_k^{-1}\left(h\deltab\taub_k+\deltab\Mb_{k-1}\right)\gb_k\nn\\
&&+h\gb_k^{-1}\left(h\taub_k+\Mb_{k-1}\right)\gb_k\deltab\Omegab_k\nn\\
&=&-h\deltab\Omegab_k\Mb_k
+\Ad^*_{\gb_k}\left(h\deltab\taub_k+\Jb\left(\deltab\Omegab_{k-1}\right)\right)
\nn\\&&+h\Mb_k\deltab\Omegab_k\nn\\
&=&\Ad_{\gb_k}^*\left(h\deltab\taub_k+\Jb\left(\deltab\Omegab_{k-1}\right)\right)
+h\left[\Mb_k,\deltab\Omegab_k\right].\nn \ee  In obtaining this expression we
used the facts that $\deltab\gb_k=h\gb_k\deltab\Omegab_k$ and
$\deltab\left(\gb_k^{-1}\right)=-h\deltab\Omegab_k\gb_k^{-1}$. The latter
equality is obtained as follows. Taking variations of
$\left(\gb_k^{\epsilon}\right)^{-1}\left(\gb_k^{\epsilon}\right)=\Ib$, we
obtain\be 0&=&
\deltab\left(\gb_k^{-1}\right)\gb_k+h\gb_k^{-1}\gb_k\deltab\Omegab_k\nn\ee which
implies that
\be\deltab\left(\gb_k^{-1}\right)=-h\deltab\Omegab_k\gb_k^{-1},\nn\ee which is
the desired result.

We now also consider the velocity boundary conditions
$\Omegab_0=\Omegab_0^*$ and $\Omegab_{N-1}=\Omegab_{N-1}^*$. Note
that variations in $\deltab\Omegab_k$ directly induce variations in
$\Wb_{k+1}$. In particular, if $k=0$ and we have the initial
constraints $\Rb_0=\Rb_0^*$ and $\Omegab_0=\Omegab_0^*$, then
$\Wb_0=0$ and $\deltab\Omegab_0=0$. Using these two equations in
equation (\ref{eq:SO3_variations_example_discrete}), we find
that\be\label{eq:variations_Wb_1} \Wb_1=0.\ee At $k=N,N-1$, the
constraints $\Rb_N=\Rb_N^*$ and $\Omegab_{N-1}=\Omegab_{N-1}^*$
imply that $\Wb_N=0$ and $\deltab\Omegab_{N-1}=0$. Using these two
equations in equation (\ref{eq:SO3_variations_example_discrete}), we
find that\be\label{eq:variations_Wb_Nm1} \Wb_{N-1}=0.\ee The
observations stated in equations (\ref{eq:variations_Wb_1}) and
(\ref{eq:variations_Wb_Nm1}) are equivalent to having
\be\label{eq:Rb_constraints_1_Nm1}\Rb_1=\Rb_0^*\exp\left(h\Omegab_0^*\right),
~\Rb_{N-1}=\Rb_N^*\exp\left(-h\Omegab_{N-1}^*\right).\ee

Taking variations of the cost functional, we obtain \be
&&\deltab\Jcal=\sum_{k=0}^N\left<\deltab\taub_k,\taub_k^{\sharp}\right>\nn\\&&+\sum_{k=0}^{N-1}\bigg<\Lambdab_k^1,-\deltab\Omegab_k
+\frac{1}{h}\left[-\Ad_{\gb_k^{-1}}\Wb_k+\Wb_{k+1}\right]\bigg>\nn\\
&&+\sum_{k=1}^{N-1}\big<\Jb\left(\deltab\Omegab_k\right)-\Ad_{\gb_k}^*\left(h\deltab\taub_k+\Jb\left(\deltab\Omegab_{k-1}\right)\right)
\nn\\&&\hspace{0.5in}-h\left[\Mb_k,\deltab\Omegab_k\right],\Lambdab_k^2\big>,\nn
\ee where we have replaced $\deltab\Mb_k$ with
$\Jb\left(\deltab\Omegab_k\right)$. Collecting terms,
 setting $\deltab\Jcal$ to zero, and using the conditions
$\Wb_0=\Wb_1=\Wb_{N-1}=\Wb_N=\deltab\Omegab_0=\deltab\Omegab_{N-1}=0$ and the
fact that $\Jb(\cdot)$ is self-adjoint, we obtain the following theorem.

\begin{thm}\label{thm:lag_disc_so3} The necessary optimality conditions for the discrete
Problem \ref{pb:SO3_discrete_example} are \be\label{eq:lag_disc_so3}
\Rb_{k+1} &=& \Rb_k\gb_k, ~k=1,\ldots,N-2\nn\\
      \Mb_k &=& \Ad_{\gb_k}^*\left(h\taub_k+\Mb_{k-1}\right),
      ~k=1,\ldots,N-1\nn\\
  0 &=& \Lambdab_{k-1}^1-\Ad_{\gb_k^{-1}}^*\Lambdab_k^1, ~k=2,\ldots,N-2\nn\\
  0
  &=&-\Lambdab_k^1+\Jb\left(\Lambdab_k^2\right)-\Jb\left(\Ad_{\gb_{k+1}}\Lambdab_{k+1}^2\right)\\&&+
  h\left[\Mb_k,\Lambdab_k^2\right], ~k=1,\ldots,N-2\nn\\
\taub_k &=& h\left(\Ad_{\gb_k}\Lambdab_k^2\right)^{\flat}, ~k=1,\ldots,N-1\nn  \\
  \Mb_k&=&\Jb\left(\Omegab_k\right), ~k=0,\ldots,N-1,\nn
\ee and the boundary conditions \be
  \Rb_0&=&\Rb_0^*, ~\Rb_1=\Rb_0^*\gb_0^*, ~\Omegab_0=\Omegab_0^*\nn\\
  \Rb_N&=&\Rb_N^*, ~\Rb_{N-1}=\Rb_N^*\left(\gb_{N-1}^*\right)^{-1}, ~\Omegab_{N-1}=\Omegab_{N-1}^*\nn\\
  \taub_0&=&\taub_N=0,\nn
\ee where $\gb_0^*=\exp(h\Omegab_0^*)$ and
$\gb_{N-1}^*=\exp\left(h\Omegab_{N-1}^*\right)$.\end{thm}

\paragraph{A Second Order Direct Approach.} Analogous to the direct
approach in continuous time, here we derive the necessary optimality conditions
in a form that does not involve the use of Lagrange multipliers. Using equation
(\ref{eq:SO3_variations_example_discrete}) and taking variation of the second
of equations (\ref{eq:SO3_eoms_discrete_example}), we obtain
\be\label{eq:variation_in_Tau}\deltab\taub_k&=&\Ad^*_{\gb_k^{-1}}\bigg(\frac{1}{h^2}\Jb\left(\Wb_{k+1}-\Ad_{\gb_k^{-1}}\Wb_k\right)
\nn\\&&
+\frac{1}{h}\left[\Wb_{k+1}-\Ad_{\gb_k^{-1}}\Wb_k,\Jb\left(\Omegab_k\right)\right]\bigg)\nn\\&&
-\frac{1}{h^2}\Jb\left(\Wb_k-\Ad_{\gb_{k-1}^{-1}}\Wb_{k-1}\right), \ee for
$k=1,\ldots,N-1$. Taking variations of the cost functional
(\ref{eq:SO3_cost_LD_discrete_example}) and substituting from equation
(\ref{eq:variation_in_Tau}) one obtains after a tedious but straight forward
computation an expression for
$\deltab\Jcal$ in terms of $\deltab\taub_k$:
\be
\deltab\Jcal&=&\sum_{k=1}^{N-1}\Bigg[\big<\Ad^*_{\gb_k^{-1}}\bigg(\frac{1}{h^2}\Jb\left(\Wb_{k+1}-\Ad_{\gb_k^{-1}}\Wb_k\right)
\nn\\&&+\frac{1}{h}\left[\Wb_{k+1}-\Ad_{\gb_k^{-1}}\Wb_k,\Jb\left(\Omegab_k\right)\right]\bigg)
-\frac{1}{h^2}\Jb\left(\Wb_k-\Ad_{\gb_{k-1}^{-1}}\Wb_{k-1}\right),\taub^{\sharp}_k\big>\Bigg]
\nn\\&&+\left<\deltab\taub_0,\taub_0^{\sharp}\right>+\left<\deltab\taub_N,\taub_N^{\sharp}\right>.\nn\ee
When $\deltab\Jcal$ is equated to zero (and after some algebraic rearrangement), one can obtain the boundary conditions on
$\taub_0,\taub_1,\taub_{N-1},\taub_{N}$ from the resulting equations below: \be \taub_0&=&0\nn\\
0&=&-\frac{1}{h^2}\left(\Jb\left(\taub_{1}^{\sharp}\right)+\Ad^*_{\gb_{1}^{-1}}\Jb\left(\Ad_{\gb_{1}^{-1}}\taub_{1}^{\sharp}\right)\right)
\nn\\&&-\frac{1}{h}\Ad^*_{\gb_{1}^{-1}}\left[\Jb\left(\Omegab_1\right),\Ad_{\gb_{1}^{-1}}\left(\taub_{1}^{\sharp}\right)\right]\nn\\
0&=&-\frac{1}{h^2}\left(\Jb\left(\taub_{N-1}^{\sharp}\right)+\Ad^*_{\gb_{N-1}^{-1}}\Jb\left(\Ad_{\gb_{N-1}^{-1}}\taub_{N-1}^{\sharp}\right)\right)
\nn\\&&-\frac{1}{h}\Ad^*_{\gb_{N-1}^{-1}}\left[\Jb\left(\Omegab_{N-1}\right),\Ad_{\gb_{N-1}^{-1}}\left(\taub_{N-1}^{\sharp}\right)\right]\nn\\
\taub_N&=&0\nn\ee as well as discrete evolution equations that are written in
algebraic nonlinear form
as:\be\label{eq:DiscreteDirect}0&=&-\frac{1}{h^2}\bigg(\Jb\left(\taub_k^{\sharp}\right)-\Ad^*_{\gb_k^{-1}}\Jb\left(\taub_{k+1}^{\sharp}\right)
\nn\\&&-\Jb\left(\Ad_{\gb_{k-1}^{-1}}\taub_{k-1}^{\sharp}\right)
+\Ad^*_{\gb_{k}^{-1}}\Jb\left(\Ad_{\gb_{k}^{-1}}\taub_{k}^{\sharp}\right)\bigg)
\nn\\&&-\frac{1}{h}\bigg(\Ad^*_{\gb_{k}^{-1}}\left[\Jb\left(\Omegab_{k}\right),\Ad_{\gb_{k}^{-1}}\left(\taub_{k}^{\sharp}\right)\right]
\nn\\&&-\frac{1}{h}\left[\Jb\left(\Omegab_{k-1}\right),\Ad_{\gb_{k-1}^{-1}}\left(\taub_{k-1}^{\sharp}\right)\right]\bigg),\ee
for $k=2,\ldots,N-2$.

The following section shows that while our discrete approximation
(\ref{eq:discrete_RB_kinematics}) is formally first-order accurate, it is
symplectically equivalent to the second-order accurate St\"ormer--Verlet
method, and hence has effective order two.

\section{Lie Symplectic Euler and Symplectic Equivalence}\label{sec:symplectic_equivalence}

Notice that the discrete Lagrangian adopted in our paper is obtained by approximating the velocity as a constant over the timestep $h$, and by approximating the integral in time by $\int_{t_1}^{t_2} f(t) dt\approx (t_2-t_1) f(t_1)$. In the Lie group setting, the constant angular velocity approximation corresponds to the condition,
\[\Rb_{k+1}=\Rb_k\exp(h\Omegab_k)\]
or equivalently,
\[\Omegab_k=\frac{1}{h}\exp^{-1}(\Rb_k^{-1}\Rb_{k+1}).\]
When we let $G=\mathbb{R}^n$, and we adopt the notation $(\qb,\vb)\in
T\mathbb{R}^n$, we obtain,
\[\vb_k=\frac{\qb_{k+1}-\qb_k}{h},\]
which is a usual finite-difference approximation for the velocity. Consider then a Lagrangian of the form,
\[L(\qb,\vb)=\frac{1}{2}\vb^T M \vb-V(\qb).\]
Approximating the action integral from $0$ to $h$ using a constant velocity
approximation and a quadrature formula, yields,
\[ \int_0^h L(\qb(t), \vb(t)) dt \approx \int_0^h L\Bigl(\qb(t),\frac{\qb_{k+1}-\qb_k}{h}\Bigr) dt \approx h L\Bigl(\qb_k,\frac{\qb_{k+1}-\qb_k}{h}\Bigr). \]
We then choose as our discrete Lagrangian,
\[ L_d(\qb_k,\qb_{k+1}) = h L\Bigl(\qb_k,\frac{\qb_{k+1}-\qb_k}{h}\Bigr) =h\left[\frac{1}{2}\Bigl(\frac{\qb_{k+1}-\qb_k}{h}\Bigr)^T M \Bigl(\frac{\qb_{k+1}-\qb_k}{h}\Bigr) - V(\qb_k)\right].\]
The discrete Euler--Lagrange equations,
\[ D_2 L_d(\qb_{k-1},\qb_k) + D_1 L_d(\qb_k, \qb_{k+1})=0,\]
yields,
\[ M\Bigl(\frac{\qb_{k}-\qb_{k-1}}{h}\Bigr)-M\Bigl(\frac{\qb_{k+1}-\qb_{k}}{h}\Bigr)-h\frac{\partial V}{\partial \qb}(\qb_k)=0,
\]
which induces an implicit update map
$(\qb_{k-1},\qb_k)\mapsto(\qb_k,\qb_{k+1})$. To obtain the corresponding
Hamiltonian update map, we push-forward this algorithm to $T^*Q$ by using the
discrete fiber derivative $\mathbb{F}L_d:Q\times Q\rightarrow T^*Q$, which
takes $(\qb_k,\qb_{k+1})\mapsto(\qb_{k+1}, D_2 L_d(\qb_k, \qb_{k+1}))$. In
particular, we have that,
\[\pb_{k+1}=D_2 L_d(\qb_k,\qb_{k+1})=M\Bigl(\frac{\qb_{k+1}-\qb_{k}}{h}\Bigr),\]
which implies
\begin{align}\label{eq:euler_position}
\qb_{k+1}=\qb_k+h M^{-1}\pb_{k+1}.
\end{align}
This allows us to rewrite the discrete Euler--Lagrange equations as,
\[
\pb_k-\pb_{k+1}-h\frac{\partial V}{\partial q}(\qb_k)=0,
\]
or equivalently,
\begin{align}\label{eq:euler_momentum}
\pb_{k+1}=\pb_k-h\frac{\partial V}{\partial q}(\qb_k).
\end{align}
Now, \eqref{eq:euler_position} and \eqref{eq:euler_momentum} are precisely the
symplectic Euler method applied to the corresponding Hamiltonian vector field,
as we shall see.

The corresponding Hamiltonian is given by,
\[H(\qb,\pb)=\frac{1}{2}\pb^T M^{-1} \pb + V(\qb).\]
Hamilton's equations yield,
\[
\begin{pmatrix}
\dot \qb\\ \dot \pb
\end{pmatrix}
=
\begin{pmatrix}
\frac{\partial H}{\partial \pb}\\
-\frac{\partial H}{\partial \qb}
\end{pmatrix}
=\begin{pmatrix}
M^{-1}\pb\\
-\frac{\partial V}{\partial \qb}
\end{pmatrix}.
\]
The symplectic Euler method has the form,
\begin{align*}
\qb_{k+1}&= \qb_k + h \dot \qb(\qb_k, \pb_{k+1}),\\
\pb_{k+1}&= \pb_k + h \dot \pb(\qb_k, \pb_{k+1}),\\
\intertext{which yields,}
\qb_{k+1}&=\qb_k+h M^{-1}\pb_{k+1},\\
\pb_{k+1}&=\pb_k+h \left(-\frac{\partial V}{\partial q}(\qb_k)\right),
\end{align*}
which is precisely what we obtained in \eqref{eq:euler_position} and
\eqref{eq:euler_momentum}. This demonstrates that our method is the
generalization of the symplectic Euler method to Lie groups, which has
important numerical consequences. While symplectic Euler is formally
first-order accurate, it is symplectically equivalent \cite{Su1993,LiSkZh1997}
to the second-order accurate St\"ormer--Verlet method \cite{HaLuWa2003}. This
means that one can obtain the St\"ormer--Verlet method $F_{\rm SV}$ by
conjugating the symplectic Euler method $F_{\rm E}$ with a symplectic
transformation $T$,
\[F_{\rm SV} = T F_{\rm E } T^{-1}.\]
In particular, numerical trajectories of symplectic Euler will shadow numerical
trajectories obtained using St\"ormer--Verlet. Consider the implications of
this symplectic equivalence for our discrete optimal control problem. Let the
boundary conditions be specified by $\qb_0, \qb_N$, and assume that we use
St\"ormer--Verlet to propagate the solution, then the boundary condition is
expressed as, $\qb_N= F_{\rm SV}^N \qb_0=(T F_{\rm E} T^{-1})^N \qb_0=T F_{\rm
E}^N T^{-1} \qb_0,$ which is equivalent to $\tilde{q}_N=T^{-1}\qb_N=F_{\rm E}^N
T^{-1}\qb_0=F_{\rm E}^N\tilde{q}_0$. This implies that if we preprocess the
boundary conditions $\qb_0, \qb_N$, to obtain
$\tilde{q}_0=T^{-1}\qb_0,\tilde{q}_N=T^{-1}\qb_N$, we could use symplectic
Euler at the internal stages to propagate the states and costates, and then
postprocess them to obtain the trajectory one would have obtained by using
St\"ormer--Verlet.

In practice, the shadowing result imparts the symplectic Euler method with the
same desirable qualitative properties as St\"ormer--Verlet, and it is not
necessary to postprocess the numerical solutions in order to achieve accurate
results. Since on an appropriate choice of charts, our Lie symplectic Euler
method reduces to symplectic Euler in coordinates, it follows that there is a
corresponding second-order Lie St\"ormer--Verlet method that our method is
symplectically equivalent to, and in particular, our method has effective order
two.

\section{Numerical Approach and Results}\la{sec:numerics}

The multiplier-free version of the first-order optimality equations,
equation (\ref{eq:DiscreteDirect}), in combination with the boundary
conditions,
\[ \Rb_0= \Rb_0^*,\;\ \Rb_N= \Rb_N^*,\;\ \Omegab_0=\Omegab_0^*,\;\ \mbox{and  }
\Omegab_{N-1}=\Omegab_{N-1}^*, \] leave the torques
$\taub_1,\ldots,\taub_{N-1}$, and the angular velocities
$\Omegab_1,\ldots,\Omegab_{N-2}$ as unknowns. By substituting the
relations $\gb_k=\exp(h\Omegab_k)$,
$\mathbf{M}_k=\mathbf{J}(\Omegab_k)$, we can rewrite the necessary
conditions (\ref{eq:DiscreteDirect}) as follows,

\begin{align*}
0=&-\frac{1}{h^2}\Big(\Jb(\taub^\sharp_k)-\Ad^*_{\exp(-h\Omegab_k)}\Jb(\taub^\sharp_{k+1})\\
&\qquad\qquad-\Jb(\Ad_{\exp(-h\Omegab_{k-1})}\taub^\sharp_{k-1}\\
&\qquad\qquad+\Ad^*_{\exp(-h\Omegab_k)}\Jb(\Ad_{\exp(-h\Omegab_k)}\taub^\sharp_k)\Big)\\
&-\frac{1}{h}\Big(\Ad^*_{\exp(-h\Omegab_k)}\Big[\Jb(\Omegab_k),\Ad_{\exp(-h\Omegab_k)}(\taub^\sharp_k)\Big]\\
&\qquad\qquad-\frac{1}{h}\Big[\Jb(\Omegab_{k-1}),\Ad_{\exp(-h\Omegab_{k-1})}(\taub^\sharp_{k-1})\Big]\Big),\\
\intertext{where $k=2,\ldots,N-2$, and the discrete evolution equations, given by line 2 of (\ref{eq:lag_disc_so3}), can be written as}
0=&\Jb(\Omegab_k)-\Ad^*_{\exp(h\Omegab_k)} (h\taub_k+\mathbf{J}(\Omegab_{k-1})),\\
\intertext{where $k=1,\ldots,N-1$. In addition, we use the boundary
conditions on $\Rb_0$ and $\Rb_N$, together with the update step given by line 1 of (\ref{eq:lag_disc_so3}) to give the last constraint,}
0=&\log\Big(\Rb_N^{-1} \Rb_0 \exp(h\Omegab_0)\ldots
\exp(h\Omegab_{N-1})\Big),
\end{align*}
where $\log$ is the logarithm map on $\SO$.

Note that while we use the direct variational approach to obtain the discrete
extremal solutions, an alternate way to obtain the discrete extremal solutions
would be to use Pontryagin's maximum principle. In particular, Bonnans and
Laurent-Varin \cite{BonL-Var:06} show that these two approaches are equivalent
in the context of symplectic partitioned Runge-Kutta schemes.

At this point it should be noted that one important advantage of the
manner in which we have discretized the optimal control problem is
that it is $\SO$-equivariant. This is to say that if we rotated all
the boundary conditions by a fixed rotation matrix, and solved the
resulting discrete optimal control problem, the solution we would
obtain would simply be the rotation of the solution of the original
problem. This can be seen quite clearly from the fact that the
discrete problem is expressed in terms of body coordinates, both in
terms of body angular velocities and body forces. In addition, the
initial and final attitudes $\Rb_0$ and $\Rb_N$ only enter in the
last equation as a relative rotation.

The $\SO$-equivariance of our numerical method is desirable, since
it ensures that our results do not depend on the choice of
coordinate frames. This is in contrast to methods based on
coordinatizing the rotation group using quaternions and Euler
angles.

Each of the equations above take values in $\mathfrak{so}(3)$.
Consider the Lie algebra isomorphism between $\mathbb{R}^3$ and
$\mathfrak{so}(3)$ given by the hat map,
\[\mathbf{v}=(v_1,v_2,v_3)\mapsto\hat{\mathbf{v}}=\begin{bmatrix}0 & -v_3 & _2\\ v_3 & 0 & -v_1\\-v_2 & v_1 & 0\end{bmatrix},\]
which maps 3-vectors to $3\times 3$ skew-symmetric matrices.
In particular, we have the following identities,
\begin{align*}
[\hat{\mathbf{u}},\hat{\mathbf{v}}]=(\mathbf{u}\times\mathbf{v})\,\hat{}\,,\quad
\Ad_\mathbf{A}\hat{\mathbf{v}}=(\mathbf{A}\mathbf{v})\,\hat{}\,.
\end{align*}
Furthermore, we identify $\mathfrak{so}(3)^*$ with $\mathbb{R}^3$ by
the usual dot product, that is to say if $\mathbf{\Pi}$,
$\mathbf{v}\in\mathbb{R}^3$, then $\langle
\mathbf{\Pi},\hat{\mathbf{v}}\rangle=\mathbf{\Pi}\cdot\mathbf{v}$.
With this identification, we have that
$\Ad^*_{\mathbf{A}^{-1}}\mathbf{\Pi}=\mathbf{A}\mathbf{\Pi}$. Using
the identities above, we write the necessary conditions using
matrix-vector products and cross products. Then, each of the
equations can be interpreted as $3$-vector valued functions, and the
system of equations can be considered as a $3(2N-3)$-vector valued
function, which is precisely the dimensionality of the unknowns.
This reduces the discrete optimal problem to a nonlinear root
finding problem.

The nonlinear system of equations was solved in MATLAB using the \texttt{fsolve} routine, where the  Jacobian is constructed column by column, and the $k$-th column is computed
using the following approximation \cite{LyMo1967},
\[ \frac{\partial\mathbf{F}}{\partial x_k}(\mathbf{x}) = \frac{1}{\epsilon}\operatorname{Im}[{\mathbf{F}(\mathbf{x}+i\epsilon\mathbf{e}_k})],\]
where $i=\sqrt{-1}$, $\mathbf{e}_k$ is a basis vector in the $x_k$
direction, and $\epsilon$ is of the order of machine epsilon. This
method is preferable to a finite-difference approximation, since it
does not suffer from round-off errors, which would otherwise limit
how small $\epsilon$ can be.

\begin{figure}[h!]
\begin{center}
\subfigure[Angular velocity and control torques]{\includegraphics[width=3in]{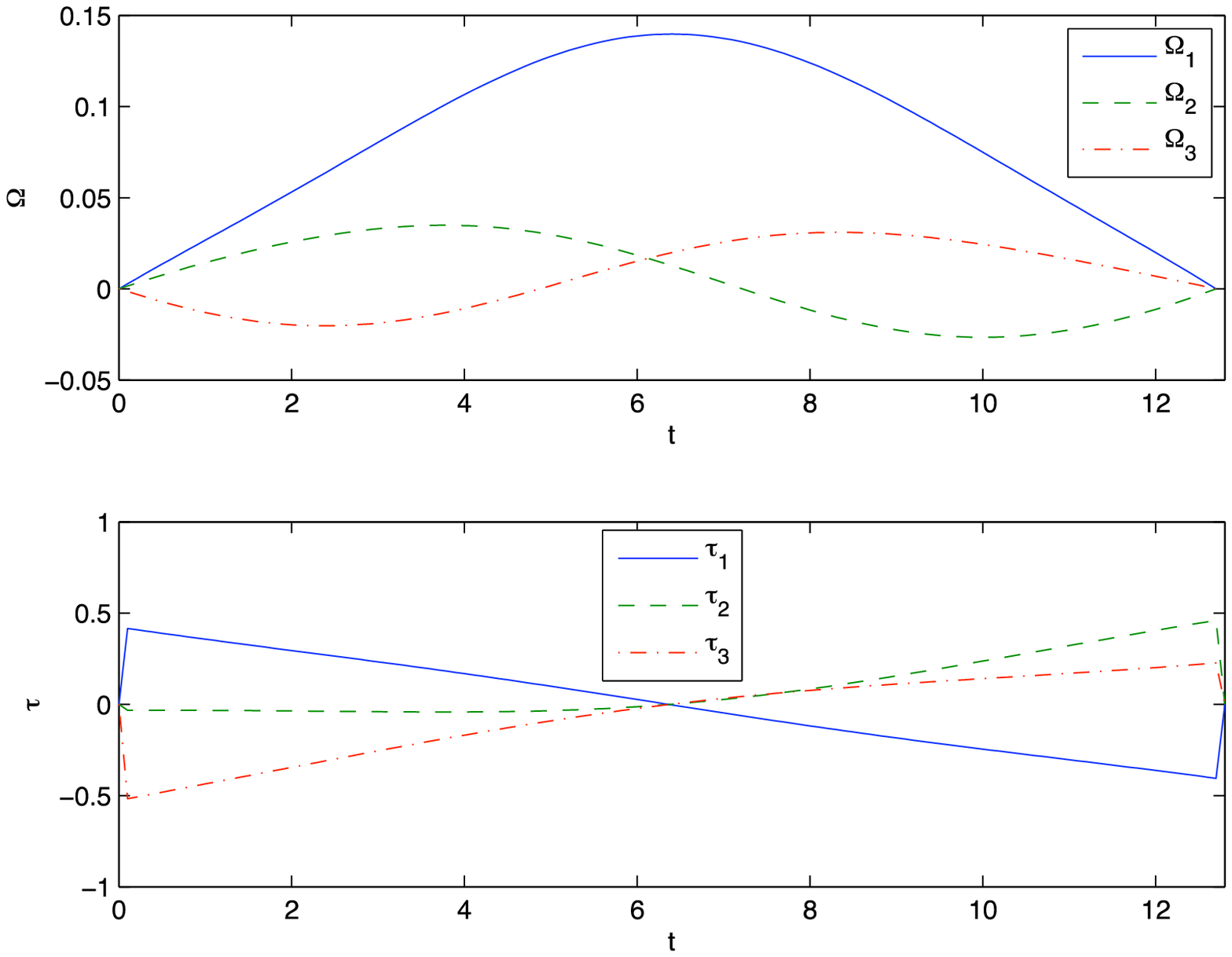}}
\subfigure[Principal axis and angle]{\includegraphics[width=3in]{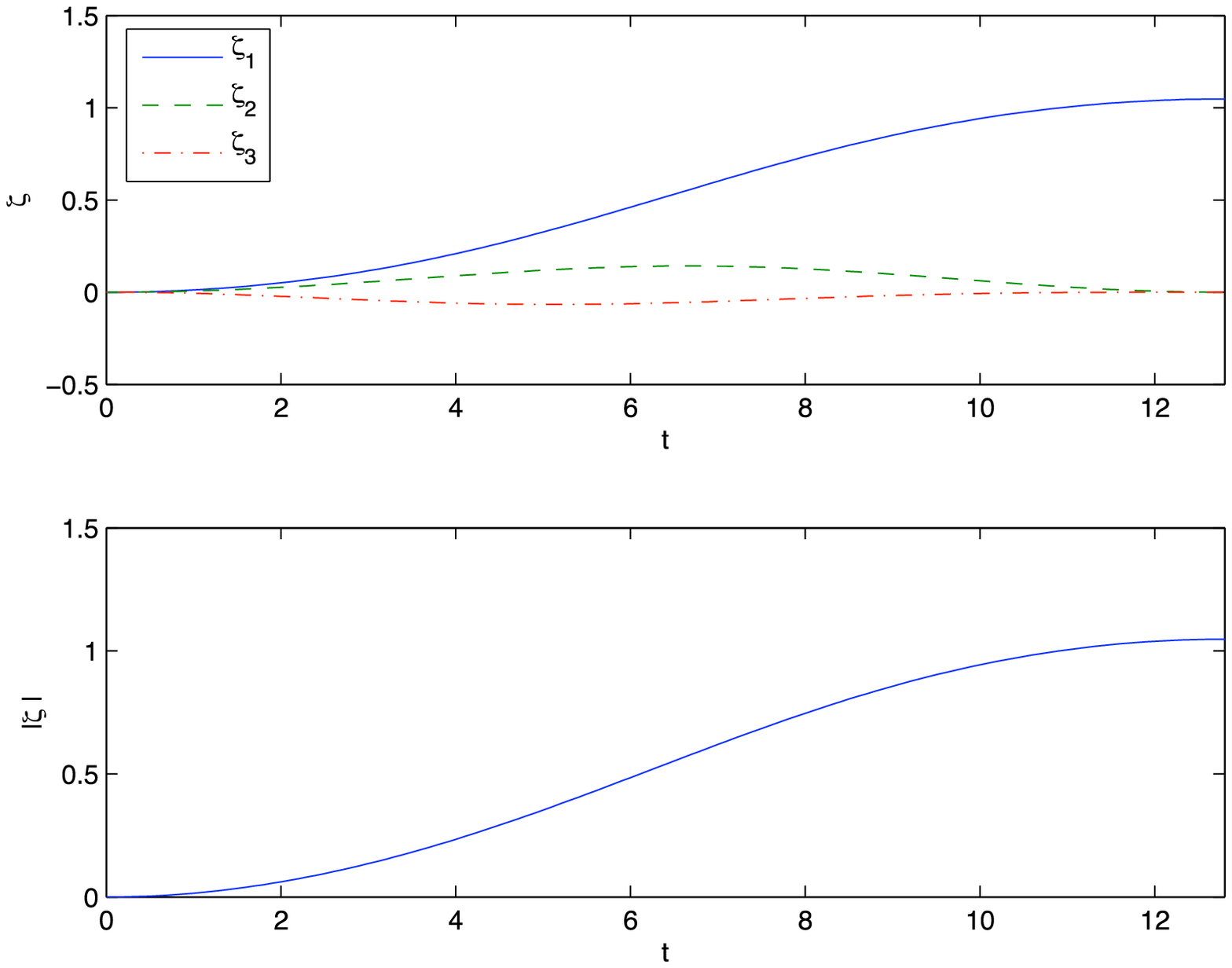}}
\subfigure[Instantaneous rotation axis]{\includegraphics[width=3in]{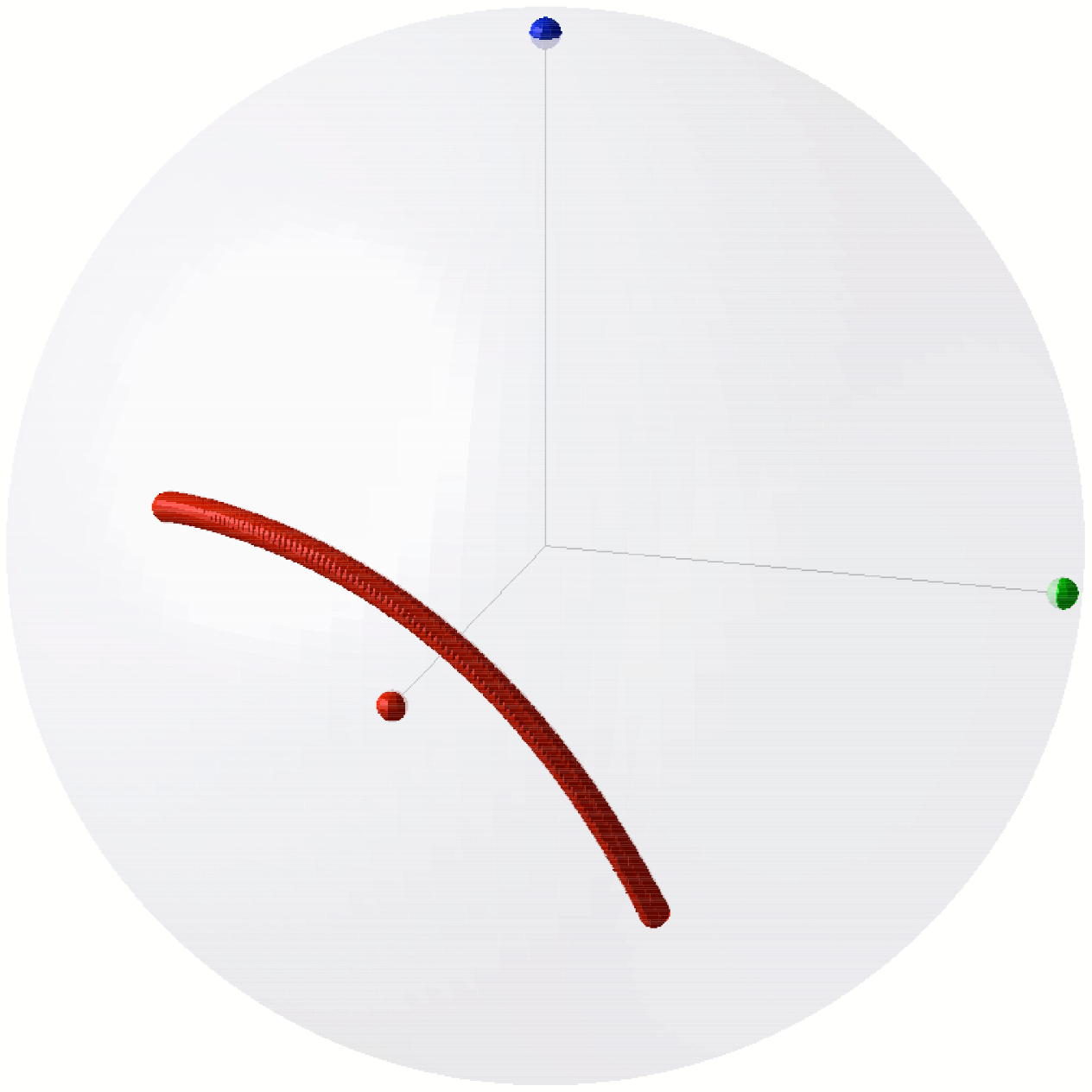}}
\caption{Discrete optimal rest-to-rest maneuver in
$\SO$.}\label{fig:result_so3}
\end{center}
\end{figure}

In our numerical simulation, we computed an optimal trajectory for a
rest-to-rest maneuver, as illustrated in Figure (\ref{fig:result_so3}). Here, the maneuver time is 12.8sec, $N=128$, and the moment of inertia is given by
\[\mathbf{J}=\begin{bmatrix}   13.25 &   -7.80 &  -11.40 \\
   -7.80 &   16.25 &    4.71 \\
  -11.40 &    4.71 &   18.37\end{bmatrix}.\]
The prescribed maneuver corresponds to a rotation by $\frac{\pi}{3}$ about the $x$-axis. Since the moment of inertia tensor is not a multiple of the identity, and the $x$-axis does not correspond to the axis of minimal inertia, the optimal trajectory does not just involve a pure rotation about the $x$-axis.
It is worth noting that the results are not rotationally symmetric about the
midpoint of the simulation interval, which is due to the fact that our choice
of update, $\Rb_{k+1}=\Rb_k\exp(h\Omegab_k)$, does not exhibit time-reversal
symmetry. In a forthcoming publication, we will introduce a reversible
algorithm to address this issue. In particular, this will involve explicitly
computing the stationarity conditions for the discrete optimal control problem
constrained by the time-symmetric Lie St\"ormer--Verlet method.

\begin{figure}[h!]
\begin{center}
\subfigure[Angular velocity and control torques]{\includegraphics[width=3in]{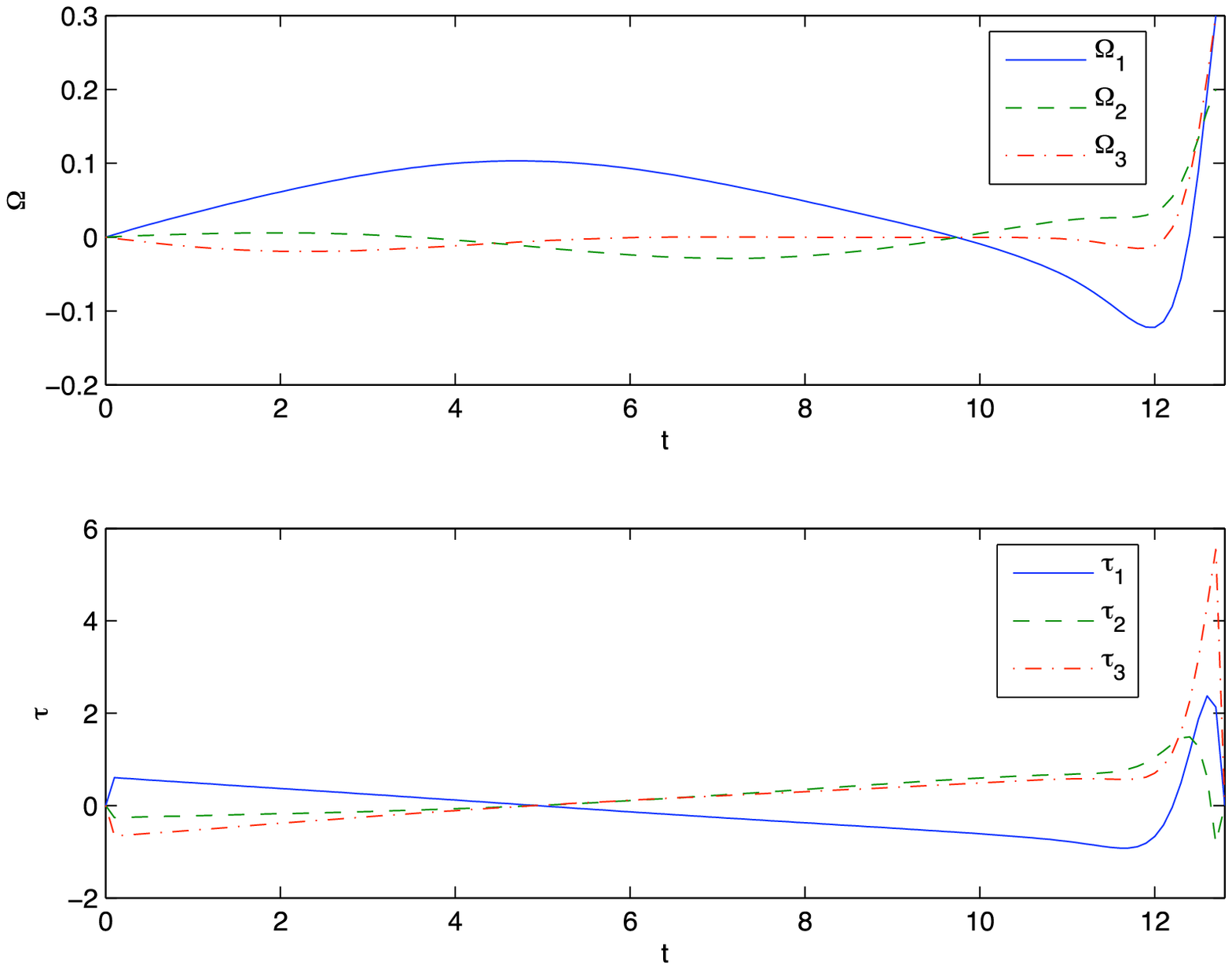}}
\subfigure[Principal axis and angle]{\includegraphics[width=3in]{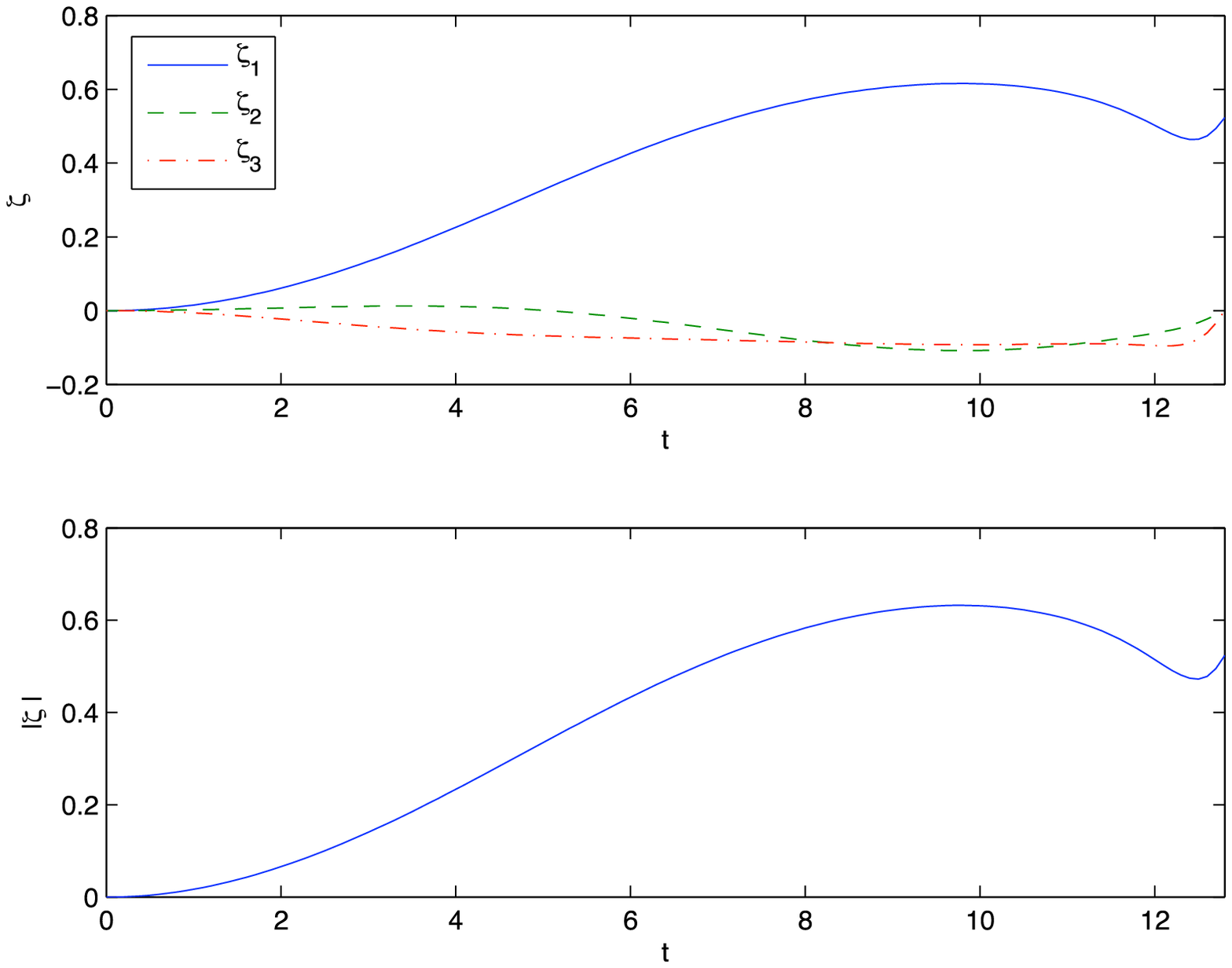}}
\subfigure[Instantaneous rotation axis]{\includegraphics[width=3in]{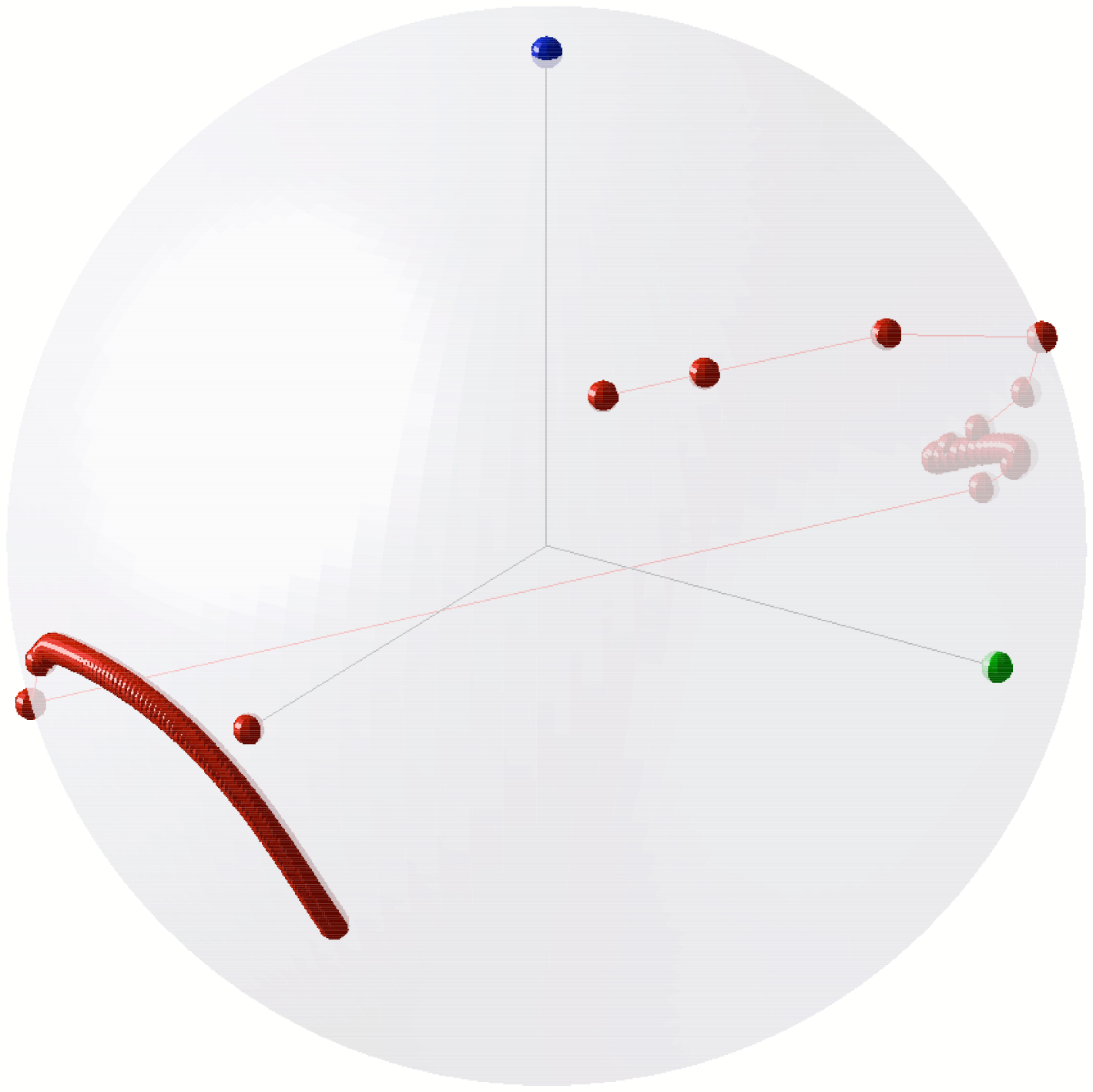}}
\caption{Discrete optimal slew-up maneuver in
$\SO$.}\label{fig:slewup_so3}
\end{center}
\end{figure}

We also present results for an optimal slew-up maneuver, illustrated in Figure (\ref{fig:slewup_so3}). This uses the same moment of inertia tensor as in the previous simulation, and the desired maneuver involves a rotation of $\frac{\pi}{6}$ about the $x$-axis from rest to a final angular velocity of $\Omega_{N-1}=\begin{bmatrix} 0.3 & 0.2 & 0.3\end{bmatrix}^{\mathrm{T}}$, over a maneuver time of 12.8sec, and $N=128$.

\section{Conclusion}\la{sec:conclusion}

In this paper we studied the continuous- and discrete-time optimal
control problem for the rigid body, where the cost to be minimized
is the external torque applied to move the rigid body from an
initial condition to some pre-specified terminal condition. In the
discrete setting, we use the discrete Lagrange--d'Alembert principle
to obtain the discrete equations of motion. The kinematics were
discretized to guarantee that the flow in phase space remains on the
Lie group $\SO$ and its algebra $\so(3)$. We described how the
necessary conditions can be solved for the general three-dimensional
case and gave a numerical example for a three-dimensional rigid body
maneuver.

The synthesis of variational mechanics with discrete-time optimal control is
particularly advantageous from the point of view of computational efficiency,
since the symplectic Euler method is symplectically conjugate to the
St\"ormer--Verlet method, and hence has effective order two. Consequently, for
our discrete-time optimal control method, the cost functional converges at a
rate which is characteristic of a second-order method, while being based on a
first-order method that is computationally cheaper.

Currently, we are investigating the use of the Pontryagin's maximum
principle with Lie group methods in continuous- and discrete-time to
obtain the necessary conditions. Additionally, we wish to generalize
the result to general Lie groups that have applications other than
the rigid body motion on $\SO$. In particular, we are interested in
controlling the motion of a rigid body \emph{in space}, which corresponds
to motion on the non-compact Lie group $\SEthree$.

\begin{center}
    \textsc{Acknowledgments}
\end{center}

The research of Islam Hussein was supported by a WPI Faculty Development Grant.
The research of Melvin Leok was partially supported by NSF grants DMS-0504747
and DMS-0726263 and a University of Michigan Rackham faculty grant. The
research of Anthony Bloch was supported by NSF grants DMS-030583, and
CMS-0408542.

\bibliographystyle{spmpsci}
\bibliography{bibliograph2}

\end{document}